\documentclass[12pt]{amsart} 
\usepackage{verbatim, latexsym, amssymb, amsmath,color}
\usepackage{enumitem}
\usepackage{epsfig}

\usepackage{geometry}
\usepackage{hyperref}

\DeclareMathOperator{\Ric}{Ric}
\DeclareMathOperator{\interior}{Int}
\DeclareMathOperator{\spt}{spt}

\DeclareMathOperator{\dmn}{dmn}

\DeclareMathOperator{\divergence}{div}

\newtheorem{theo}{Theorem}[]
\newtheorem{prop}[theo]{Proposition}
\newtheorem{lemme}[theo]{Lemma}
\newtheorem{definition}[theo]{Definition}

\newtheorem{remarque}[theo]{Remark}

\begin{document}
\title[Existence of infinitely many minimal hypersurfaces]
{Existence of infinitely many minimal hypersurfaces in closed manifolds}

\author{Antoine Song}
\address{California Institute of Technology\\ 177 Linde Hall, \#1200 E. California Blvd., Pasadena, CA 91125}
\email{aysong@caltech.edu}
\thanks{The author was partially supported by NSF-DMS-1509027.}

\maketitle

\begin{abstract}
Using min-max theory, we show that in any closed Riemannian manifold of dimension at least $3$ and at most $7$, there exist infinitely many smoothly embedded closed minimal hypersurfaces. It proves a conjecture of S.-T. Yau. This paper builds on the methods developed by F. C. Marques and A. Neves.

\end{abstract}

\section{Introduction}

An important theme in Riemannian geometry is the study of the space of $k$-dimensional submanifolds in a Riemannian manifold with the $k$-volume functional by means of variational methods. For instance in the case of codimension one, a closed hypersurface $\Sigma$ in an $(n+1)$-manifold $(M,g)$ is a minimal hypersurface when it is a critical point of the $n$-volume functional, and a natural question which arises is how to construct these special objects. Birkhoff \cite{Birkhoff} was the first to use a mountain-pass type argument to construct a closed geodesic (which is a minimal hypersurface of dimension one) in any $2$-sphere. In higher dimension, Almgren and Pitts developed a similar but much more complex theory in order to construct minimal hypersurfaces: they showed that in any closed Riemannian manifold of dimension at least $3$ and less than $7$, there exists a smoothly embedded closed minimal hypersurface \cite{P}. This result was shortly after extended to the case of dimension $7$ by Schoen and Simon \cite{SchoenSimon}. For dimensions higher than $7$, there still exists a minimal hypersurface, but it may be singular along a set of Hausdorff codimension at least $7$. Motivated by these results, in the early 80's Yau formulated the following conjecture.
\newline
\\
\textbf{Conjecture} (S.-T. Yau \cite{Yauproblemsection}):
\textit{Any closed three-dimensional manifold must contain an infinite number of immersed minimal surfaces.}
\newline
\\

A recurring difficulty when dealing with minimal hypersurfaces produced by min-max theory is the issue of multiplicity. The same underlying $n$-dimensional connected minimal hypersurface endowed with different integer multiplicities represents different critical points of the $n$-volume functional from the variational point of view, which is problematic if one wants to find geometrically distinct minimal hypersurfaces.

Even in the case of geodesics ($n=1$), the analogous conjecture was settled only later, by the combined work of Franks \cite{Franks} and Bangert \cite{Bangert}. The first mentioned paper relies on dynamical systems methods while the second one uses Morse theoretic arguments specific to closed curves in surfaces. 

Up until a few years ago, general cases of the conjecture in higher dimensions remained elusive. In a paper which gave the impulsion to most of the recent activity on this question, Marques and Neves \cite{MaNeinfinity} proved the conjecture of Yau in the case where $(M^{n+1},g)$ is closed, $2\leq n\leq 6$ and $M$ satisfies the so-called Frankel property: any two closed embedded minimal hypersurfaces intersect each other. For instance this property is implied by $\Ric_g>0$. In their proof, they recast the sublinear bounds for a certain sequence of min-max widths, proved by Gromov \cite{Gromovwaist} and Guth \cite{Guth}, in the context of an extended Almgren-Pitts' min-max theory and argue as follows. If the Frankel property is satisfied, then each width is realized as an integer multiple of the area of a connected minimal hypersurface. But since the widths grow sublinearly, when they are all distinct a counting argument leads to a contradiction if there were only finitely many minimal hypersurfaces to begin with. When two widths coincide then a Lusternick-Schnirelmann type argument gives the result. Following this work, Marques and Neves went on with the development of a higher dimensional Morse theory for the $n$-volume functional \cite{MaNeindexbound} \cite{MaNemultiplicityone}.
 
Afterwards, Irie, Marques and Neves \cite{IrieMaNe} settled the conjecture for generic metrics (in the Baire sense) by showing that for generic metrics on a closed manifold with the usual restriction on the dimension, the union of minimal hypersurfaces is dense. This was later quantified into a generic equidistribution result in \cite{MaNeSong} by Marques, Neves and the author. The central ingredient of their proof is the Weyl law for the volume spectrum proved by Liokumovich, Marques and Neves \cite{LioMaNe}. 

Previously, Marques and Neves had devised another strategy to tackle the question: if for generic metrics the multiplicity of the minimal hypersurfaces produced by min-max was one, then since the widths are a sequence going to infinity, it would automatically lead to another proof of Yau's conjecture in the generic case. This is the Multiplicity One conjecture \cite{MaNeindexbound}, which was confirmed by X. Zhou \cite{Zhoumultiplicityone} in dimensions $3$ to $7$ . In dimension $3$, that was previously proved by Chodosh and Mantoulidis \cite{ChodMant} for bumpy metrics in the context of the Allen-Cahn version of min-max theory. (The latter was an alternative to the Almgren-Pitts theory, proposed by Guaraco \cite{Guaraco} and later extended by Gaspar and Guaraco \cite{GasGua}.)

The goal of this paper is to prove Yau's conjecture in the general case where the metric may not be generic. Our proof builds on \cite{MaNeinfinity}.

\begin{theo}\label{Yau}
In any closed Riemannian manifold of dimension at least $3$ and at most $7$, there exist infinitely many smoothly embedded closed minimal hypersurfaces.
\end{theo}

Actually as a byproduct of our proof, we produce many closed embedded minimal hypersurfaces \textit{locally}: for instance if the metric is bumpy, then any compact manifold (of dimension $3\leq n+1 \leq 7$) with stable boundary contains infinitely many closed minimal hypersurfaces. A more precise formulation is stated in Remark \ref{remarque}.

We note that in the special case of closed hyperbolic $3$-manifolds, the incompressible surfaces found by Kahn and Markovi\'{c} \cite{KahnMarkovic} produce after minimization an infinite number of closed immersed minimal surfaces. Let us finally mention other approaches which have been so far less successful in handling general cases of the conjecture. In \cite{Rubinsteinnotes}, Rubinstein sketched an argument to produce infinitely many minimally immersed surfaces in any hyperbolic $3$-manifold with finite volume. He relies among other things on a construction of certain minimal surfaces of index at most one using Heegaard splittings via the smooth one-parameter min-max theory of Simon and Smith \cite{Smith}, which we justified rigorously in \cite{KetLioSong} jointly with Ketover and Liokumovich. Kapouleas \cite{Kapouleas1} \cite{Kapouleas2} proposed a strategy to construct infinitely many different embedded minimal surfaces in $3$-manifolds with a generic metric by desingularizing two intersecting minimal surfaces or by gluing small necks between two surfaces close to an unstable minimal surface.

\vspace{1em}
\textbf{Remark.}
Since the first version of this article was completed in 2018, several directly related papers appeared. We mention for instance \cite{Zhoumultiplicityone,Antoine19,Li19,CM21,Li20,Song-Zhou20,WangZhichao20,StevensSun21}. We used \cite{Zhoumultiplicityone,GLWZ21,SWZ20} to simplify some of the original arguments.

\subsection*{Outline of proof:}

To simplify the discussion, let us consider a closed manifold $M^{n+1}$ ($2\leq n\leq 6$) endowed with a bumpy metric $g$ and not containing one-sided minimal hypersurfaces. Suppose towards a contradiction that it has only finitely many minimal hypersurfaces. Then by the previous work of Marques and Neves \cite{MaNeinfinity}, there exists at least one stable minimal hypersurface. We start by cutting $M$ along stable minimal hypersurfaces a maximal number of times, keeping one component at each step, to get a connected compact ``core" $U$ with stable minimal boundary. Then one can show that any two minimal hypersurfaces in the interior of $U$ intersect, i.e. the core satisfies the Frankel property. The main thrust of the proof is then to produce by min-max theory closed minimal hypersurfaces confined inside the interior of $U$. The difficulty of this is that while there is a version of min-max theory for manifolds with boundary due to M. Li and X. Zhou \cite{LiZhoufree}, the minimal hypersurfaces obtained have themselves non-empty boundaries in general. We are able to produce closed minimal hypersurfaces by considering the non-compact manifold with cylindrical ends $\mathcal{C}(U)$ obtained by gluing to $U$ the straight cylindrical manifold $\partial U \times [0,\infty)$ along the boundary $\partial U$. Note that the metric on $\mathcal{C}(U)$ may not be smooth. We will prove that min-max theory applied to the non-compact manifold $\mathcal{C}(U)$ associates to each min-max widths $\omega_p=\omega_p(\mathcal{C}(U))$ an integer multiple of one closed connected minimal hypersurface in the interior of the core $U$, whose total $n$-volume is $\omega_p$. The next step is to prove that, in contrast to the sublinear bounds of Gromov-Guth in the compact case \cite{Gromovwaist}\cite{Guth}, the widths $\omega_p$ here behave asymptotically linearly and the leading factor is the $n$-volume of the largest component $\Sigma_1$ of $\partial U$. Since all the closed minimal hypersurfaces inside $U\backslash \partial U$ (which we assumed are in finite number) have their $n$-volume larger than that of $\Sigma_1$, an arithmetic lemma shows that the widths $\omega_p$ eventually become much larger than $p$ times the $n$-volume of $\Sigma_1$, a contradiction. These arguments do not rely on the bumpiness of the metric. In general to obtain the core we will only cut along the stable minimal hypersurfaces which are area minimizing at least on one side.

\subsection*{Acknowledgements:} I am very grateful to my advisor Fernando Cod\'{a} Marques for his constant support, his generosity and inspiring discussions during the course of this work. I also thank him for pointing out references \cite{WhiteMP} and \cite{BrayBrendleNeves}. I would like to thank Andr\'{e} Neves for many valuable conversations.

\section{Min-max theory on a non-compact manifold with cylindrical ends} 
All manifolds considered in this paper
have dimension $n+1$ with $3\leq n+1\leq 7$. The minimal hypersurfaces are smooth embedded. In this section, we construct a certain non-compact manifold with cylindrical ends, and explain how to produce closed minimal hypersurfaces by min-max theory. Their areas satisfy a ``cylindrical Weyl law'', i.e. a linear asymptotic behavior depending on the ends.

Before starting, we point out that previously there were already some results on min-max minimal hypersurfaces in non-compact manifolds. In \cite{GromovPlateauStein}, Gromov constructed minimal hypersurfaces in some special classes of non-compact manifolds. The existence of a closed embedded minimal surface in any finite volume hyperbolic $3$-manifold was proved by Collin-Hauswirth-Mazet-Rosenberg in \cite{ColHauMazRos}. Montezuma showed in \cite{Montezuma1} the existence of one closed minimal hypersurface in manifolds which contain a mean-concave domain and whose ends have controlled geometry. 
In \cite{ChamLiok}, Chambers and Liokumovich produced a complete minimal hypersurface in any finite volume manifolds; see \cite{Antoine19} for a simple alternative proof. Our situation is very different from the previous ones, in that our goal is to localize precisely the min-max minimal hypersurfaces, which should lie inside a given compact domain. In particular, we have to choose a potentially non-smooth metric in order to force them to stay in that domain.

\subsection{Preliminaries in geometric mesure theory}

We recall some definitions about varifolds (see \cite{Simon}) and observe that a version of the usual monotonicity formula holds uniformly for stationary varifolds in metrics close to the Euclidean metric in the $C^1$ norm. As a consequence, we prove an integrality result for certain varifolds which are limits of integral varifolds stationary for different metrics close in the $C^0$ norm and bounded in the $C^1$ topology.

Let $\Omega$ be an open subset of the Euclidean space $\mathbb{R}^{n+1}$. Let $\mathbf{V}_n(\Omega)$ be the set of $n$-varifolds in $\Omega$, i.e. the set of Radon measures on the Grassmannian $\Omega\times\mathbf{G}(n+1,n)$, where $ \mathbf{G}(n+1,n)$ denotes the family of unoriented $n$-dimensional subspaces in $\mathbb{R}^{n+1}$. More generally by abuse of notation, if $\Omega$ is an open $(n+1)$-dimensional manifold, we will still write $\Omega\times\mathbf{G}(n+1,n)$ to denote the Grassmannian bundle of $n$-planes over $\Omega$. In this setting, one defines a varifold similarly. The weight of a varifold $V\in \mathbf{V}_n(\Omega)$ is denoted by $\|V\|$ and its mass is $\mathbf{M}(V)=\|V\|(\Omega)$. If $A$ is a Borel subset of $\Omega$, then $V\llcorner A$ is the restriction of $V$ to $A\times \mathbf{G}(n+1,n)$. The support of $V$, $\spt(V)$, is the smallest relatively closed subset $B\subset \Omega$ such that $V\llcorner(\Omega \backslash B) =0$. We will be mostly interested in rectifiable $n$-varifolds $V$. They are determined by an $n$-dimensional rectifiable set $W$ and a density function $\Theta(.) = \Theta^n(dV,.)$, i.e. a nonnegative function measurable and locally integrable with respect to the $n$-dimensional Hausdorff measure $\mathcal{H}^n$. For Borel set $A\subset \Omega\times\mathbf{G}(n+1,n)$, 
$$V(A):=\int_{\{x\in W; (x,T_xW) \in A\}} \Theta d\mathcal{H}^n(x).$$
If the density function of $V$ takes integer values, then $V$ is said to be integral. Again these definitions extend more generally to an open manifold $\Omega$ endowed with a general metric $g$. For instance let $g_1$, $g_2$ be two smooth metrics on an open domain $\Omega$ of the Euclidean $n$-space, and let $\mathcal{H}^n_1$, $\mathcal{H}^n_{2}$ be the respective Hausdorff measures. If $V$ is a rectifiable varifold in $\Omega$, and $\Theta_1$, $\Theta_2$ are the density functions of $V$ respectively for $g_1$, $g_2$, then 
$$\Theta_1\mathcal{H}^n_1 = \Theta_2\mathcal{H}^n_{2}$$
and if $S$ is an $n$-plane and $\det_S g_i$ is the determinant of the restriction of $g_i$ to $S$ in an orthonormal basis of $S$ for the Euclidean metric, we have at $dV$-a.e. $x\in W=\spt V$:
$$
\Theta_2(x)=\sqrt{\frac{\det_{T_xW}g_1}{\det_{T_xW}g_2}}\Theta_1(x).
$$
A varifold $V$ is called stationary for a metric $\tilde{g}$ (or $\tilde{g}$-stationary) in $\Omega$ when for all smooth vector field $X$ with compact support in $\Omega$, 
$$\delta V(X):= \int_{\Omega \times \mathbf{G}(n+1,n)} \tilde{\divergence}_S X \quad dV(x,S) =0.$$
Here $\tilde{\divergence}_S=\sum_{i=1}^n \langle \tilde{\nabla}_{e_i} X, e_i\rangle$, where $\tilde{\nabla}$ is the Levi-Civita connection of $\tilde{g}$ and $\{e_1,\dots,e_n\}$ is any $\tilde{g}$-orthonormal basis of the $n$-subspace $S$. Note that the notion of rectifiability does not depend on the metric while being integral or stationary depends strongly on the metric.

Let $B_r$ be the Euclidean ball of radius $r$ in $\mathbb{R}^{n+1}$, and let $g_{eucl}$ be the standard Euclidean metric. We start by stating the known fact that the monotonicity formula holds uniformly for stationary varifolds in metrics which are bounded in the $C^1$ topology. We leave the proof to the Appendix.

\begin{lemme} \label{monotonicity}
Let $\eta\in(0,1)$. Consider a metric $\tilde{g}$ on $B_3$ such that 
$$\|\tilde{g}-g_{eucl}\|_{C^1(\Omega)}\leq \eta.$$
Then there exist nonnegative real numbers $\mathfrak{c}=\mathfrak{c}(\eta)$ and $\mathfrak{a}=\mathfrak{a}(\eta)$ such that the following holds. Let ${V}$ be a $\tilde{g}$-stationary $n$-dimensional varifold in $B_3$.
Then for any $\xi\in B_1$, for any $0<\sigma\leq \rho < 1$:
$$\exp({\mathfrak{c}\sigma})\frac{\|V\|(B(\xi,\sigma))}{\sigma^n} \leq (1+\mathfrak{a})\exp({\mathfrak{c}(1+\frak{a})\rho}) \frac{\|V\|(B(\xi,(1+\frak{a})\rho))}{\rho^n},$$
where $B(\xi,r)$ is the Euclidean ball of radius $r$ centered at $\xi$. More generally, for any smooth function $h:B_3\to [0,1]$ and $\xi$, $\sigma$, $\rho$ as above,
\begin{align*}
& \exp({\mathfrak{c}\sigma} ) \frac{1}{\sigma^n} \int_{B(\xi,\sigma)} h\quad d\|V\|  \\ &\leq  
(1+\mathfrak{a})\exp({\mathfrak{c}(1+\frak{a})\rho}) \Big[ \frac{1}{\rho^n} \int_{B(\xi,(1+\frak{a})\rho)} h\quad d\|V\| \\ 
&\quad + \int_\sigma^{(1+\frak{a})\rho} \tau^{-n}\int_{B(\xi,\tau)\times \mathbf{G}(n+1,n)} |\nabla_Sh| \quad dV(x,S) d\tau\Big]
\end{align*}
where $|{\nabla}_Sh|$ is the norm of the gradient of $h$ along $S$ computed with $g_{eucl}$. Moreover, the constants $\mathfrak{c}(\eta)$ and $\mathfrak{a}(\eta)$ converge to $0$ with $\eta$:
$$\lim_{\eta\to 0} \mathfrak{c} = \lim_{\eta\to 0} \mathfrak{a} = 0.$$
\end{lemme}


The following proposition, which will be useful in Subsection \ref{confined}, is a consequence of the proof of Allard's compactness theorem for stationary integral varifolds. 

\begin{prop}\label{Allard}
Fix $\eta>0$. Let $g$ and $\{g_k\}_k\geq0$ be metrics on an open domain $\Omega \subset \mathbb{R}^{n+1}$ such that 
\begin{equation} \label{gmt}
 \|g_k\|_{C^1(\Omega)} \leq \eta. 
 \end{equation}
Let $\{V_k\}$ be a sequence of $n$-dimensional varifolds which are integral and stationary with respect to  $g_k$. Suppose that $V_k$ converges to an $n$-dimensional varifold $V$ in the varifold sense, and that $V$ is $g$-stationary. Suppose moreover that 
\begin{equation}\label{hola}
\forall x\in \spt(V), \quad \lim_{k\to\infty} |g_k(x)-g(x)|=0.
\end{equation}
Then $V$ is integral.
\end{prop}

\begin{proof}
Before starting the proof, let us give some intuition as to why the statement is expected. Recall that the classical Allard's compactness theorem \cite{Allardfirstvar} says that given a sequence of $n$-varifolds $V_k$ converging to a stationary varifold $V$ inside a fixed Riemannian manifold $(M,g)$, then $V$ must be an integral varifold provided that 
\begin{itemize}
\item each varifold $V_k$ is $g$-integral,
\item the first variations of $V_k$ with respect to $g$ are uniformly bounded as $k\to \infty$.
\end{itemize}
Proposition \ref{Allard} is similar, since we just replace the two bullets above by
\begin{itemize}
\item each varifold $V_k$ is $g_k$-integral and $g_k$ converges to $g$ on $\spt(V)$ in the $C^0$-norm,
\item each $V_k$ is $g_k$-stationary and $g_k$ have uniformly bounded $C^1$-norm  as $k\to \infty$.
\end{itemize}
\vspace{1em}

The proof of the proposition is based on two easy facts.

\textbf{Fact 1.} Let $x\in \Omega$, $B\subset \Omega$ a Euclidean ball centered at $x$. Let $\eta_{\lambda}$ be the rescaling map of factor $\lambda>0$ centered at $x$: for all $y\in B$, ${\eta}_{\lambda}(y)=\lambda (y-x) +x$. By the $C^1$ bound on the metrics $g_k$, we have:
$$\lim_{\lambda \to\infty} \big(\sup_{B,k}|\nabla \lambda^2\eta_{\lambda^{-1}}^*g_k|\big)=0$$
where the norm and $\nabla$ are computed with the Euclidean metric on $B$. So the metrics $ \lambda^2\eta_{\lambda^{-1}}^*g_k$ become arbitrarily close to a flat one in the $C^1$ topology in $B$, as $\lambda\to \infty$. Moreover if $x\in \spt(V)$, then by (\ref{hola}) $\lambda^2\eta_{\lambda^{-1}}^*g$ and $ \lambda^2\eta_{\lambda^{-1}}^*g_k$ are all close to the same flat metric equal to the constant metric $g_{const} \equiv g(x)$. This enables us to apply Lemma \ref{monotonicity} to the rescaled metrics.

\vspace{1em}
\textbf{Fact 2.} Note that $V_k$ are not necessarily integral for the metric $g$. Nevertheless, by the $C^1$ bound (\ref{gmt}) and our assumption (\ref{hola}), for $dV_k$-a.e. $x\in \spt(V_k)$ at distance less than $d$ from $\spt(V)$, if $\Theta^n(dV_k,x)$ is the density of $V_k$ at $x$ computed with $g$, then there is an integer $D$ with
$$|\Theta^n(dV_k,x)-D|\leq \hat{\epsilon}(d,k) \Theta^n(dV_k,x),$$
where $\hat{\epsilon}$ converges to $0$ as $d\to0$, $k\to \infty$.

\vspace{1em}
We first show that the $g$-stationary limit $V$ is rectifiable. Because of the rectifiability theorem for varifolds with bounded first variation \cite[Theorem 5.5]{Allardfirstvar}\cite[Chapter 8 Theorem 42.4]{Simon}, we only need to check that at any point $x\in\spt(V)$, the density $\Theta^n(dV,x)$ computed with $g$ is strictly positive. Let $r>0$. By varifold convergence of $V_k$ to $V$, by (\ref{gmt}) and (\ref{hola}), for $k$ large enough the $g$-ball $B_g(x,r)$ contains the $g_k$-ball $B_{g_k}(x_k, \alpha r)$ for some sequence $x_k\in \spt(V_k)$ converging to $x$ and some $\alpha\in (0,1)$ independent of $r$ and $k$. By the monotonicity formula Lemma \ref{monotonicity} and Fact 1 applied to the $g_k$-stationary integral varifolds $V_k$, we have $\|V_k\|(B_{g_k}(x_k,\alpha r))\geq cr^n$ for some $c>0$ independent of $r$ and $k$. By varifold convergence again, we get 
$\|V\|(B_{g}(x,2r))\geq {c}r^n$, which is enough to conclude $\Theta^n(dV,x)>0$.


We finally need to verify that $V$ is integral, namely its density function takes integer values. We check that in the proof of integrality in the usual Allard's compactness theorem for integral stationary varifolds \cite[Theorem 6.4]{Allardfirstvar} (see also \cite[Chapter 8 Section 42]{Simon}), the stationarity assumption comes into play only in the form of one of its consequence, the monotonicity formula. 
Here even though the $V_k$ are not stationary for $g$, a version of the monotonicity formula, Lemma \ref{monotonicity}, still holds. Hence, making the necessary modifications in the proof of \cite[Theorem 6.4]{Allardfirstvar} using Fact 1 and remembering Fact 2, we conclude that $V$ is integral.

\end{proof}

\subsection{Construction of a non-compact manifold with cylindrical ends} \label{construction}
Let $(U,g)$ be a connected compact Riemannian manifold with boundary endowed with a smooth metric $g$. Suppose that $\partial U$ is a minimal surface such that a neighborhood of $\partial U$ in $U$ is smoothly foliated by closed leaves whose mean curvature vectors are pointing towards $\partial U$. In other words, we assume that there is a diffeomorphism
$$\Phi : \partial U \times [0,\hat{t}] \to U$$
where $\Phi(\partial U \times \{0\}) = \partial U$ is a minimal hypersurface, and for all $t\in(0,\hat{t}]$, the leaf $\Phi(\partial U \times \{t\})$ has non-zero mean curvature vector pointing towards $\partial U$. By convention, $\Phi(\partial U \times \{t\})$ has positive mean curvature with respect to the normal vector in the direction given by $\frac{\partial}{\partial t}$.

Let $\varphi: \partial U \times \{0\} \to \partial U$ be the canonical identifying map. Define the following non-compact manifold with cylindrical ends:
$$\mathcal{C}(U):= U \cup_\varphi (\partial U \times [0,\infty)).$$
We endow it with the metric $h$ such that $h=g$ on $U$ and $h=g\llcorner{\partial U} \oplus ds^2$. Here $g\llcorner{\partial U}$ is the restriction of $g$ to the tangent bundle of the boundary $\partial U$ and $g\llcorner{\partial U} \oplus ds^2$ is the product metric on $\partial U \times [0,\infty)$. Note that the metric $h$ is only Lipschitz continuous in general. 

Next, we define for any small $\epsilon>0$ a compact Riemannian manifold with boundary $(U_\epsilon,h_\epsilon)$ diffeomorphic to $U$ and converging to $(\mathcal{C}(U),h)$ as $\epsilon\to 0$ in a sense to be defined later. These $U_\epsilon$ are bigger and bigger chopped pieces of $\mathcal{C}(U)$ with a slightly modified smoothed metric $h_\epsilon$. For $0<\epsilon<\hat{t}$, let us first define  
$$\tilde{U}_\epsilon = U \backslash \Phi(\partial U \times [0,\epsilon)).$$
Its boundary is $\partial \tilde{U}_\epsilon= \Phi(\partial U \times \{\epsilon\})$. For a small positive number $\delta_{\epsilon}>0$, the following map 
$$\tilde{\gamma}_\epsilon : \partial \tilde{U}_\epsilon \times [-\delta_{\epsilon},0] \to M$$
$$\tilde{\gamma}_\epsilon(x,t) = \exp({x}, t \nu)$$
is well-defined and gives Fermi coordinates on one side of $\partial \tilde{U}_\epsilon$. By ``coordinates", we mean that $\tilde{\gamma}_\epsilon$ is a diffeomorphism onto its image. Here $\exp$ is the exponential map with respect to $g$, $\nu$ is the inward unit normal of $\partial \tilde{U}_\epsilon$. Note that $\tilde\gamma_\epsilon(\partial \tilde{U}_\epsilon\times \{0\}) = \partial \tilde{U}_\epsilon$ and $\tilde\gamma_\epsilon(\partial \tilde{U}_\epsilon\times [-\delta_{\epsilon},0]) \subset M\backslash\tilde{U}_\epsilon$.
We can moreover suppose $\delta_{\epsilon}>0$ small enough so that 
\begin{itemize}
\item $\lim_{\epsilon\to 0} \delta_\epsilon=0,$
\item $\tilde{\gamma}_\epsilon(\partial \tilde{U}_\epsilon \times [-\delta_{\epsilon},0]) \subset \Phi(\partial U \times [0,\epsilon])$, 
\item and for all $t\in [-\delta_{\epsilon},0]$, the hypersurface $\tilde{\gamma}_\epsilon( \partial \tilde{U}_\epsilon \times\{t\})$ has positive mean curvature. (The normal vector is in the direction of $\nu$.) 
\end{itemize}
For $0 < \epsilon<\hat{t}$, we define
$${U}_\epsilon = \tilde{U}_\epsilon \cup \tilde{\gamma}_\epsilon (\partial \tilde{U}_\epsilon \times [-\delta_{\epsilon},0])$$
and by convention $U_0 = U$.

For a positive number $\hat{d}>0$ small enough and fixed independently of $\epsilon\in[0,\hat{t})$, consider the following Fermi coordinates on a $\hat{d}$-neighborhood of $\partial {U}_\epsilon$:
$$\gamma_\epsilon : \partial {U}_\epsilon \times [-\hat{d},\hat{d}] \to M$$
$$\gamma_\epsilon(x,t) = \exp({x}, t \nu)$$
where $\exp$ is the exponential map with respect to $g$, $\nu$ is the inward unit normal of $\partial {U}_\epsilon$.
We remark for clarity that for all $s\in [0,\delta_\epsilon]$, 
$$\gamma_\epsilon (\partial {U}_\epsilon\times \{s\}) = \tilde\gamma_\epsilon (\partial \tilde{U}_\epsilon\times \{-\delta_\epsilon+s\}).$$ For each small positive $\epsilon$, choose a smooth function $\vartheta_\epsilon:[0,\delta_\epsilon] \to \mathbb{R}$ and $z_\epsilon\in(0,\delta_\epsilon)$ with the following properties:
\begin{itemize}
\item $1\leq \vartheta_\epsilon$ and $\frac{\partial}{\partial t}\vartheta_\epsilon \leq 0$,
\item $\vartheta_\epsilon\equiv1$ in a neighborhood of $\delta_\epsilon$,
\item $\vartheta_\epsilon$ is constant on $[0,z_\epsilon]$,
\item $\lim_{\epsilon\to 0}\int_{[0,\delta_\epsilon]}\vartheta_\epsilon =\infty$,
\item $\lim_{\epsilon\to 0}\int_{[z_\epsilon,\delta_\epsilon]}\vartheta_\epsilon =0$.
\end{itemize}
This function naturally induces a function on $U_\epsilon$ still called $\vartheta_\epsilon$, defined by
$$\vartheta_\epsilon(\gamma_\epsilon(x,t)) = \vartheta_\epsilon(t) \text{ for all } (x,t)\in \partial U_\epsilon \times [0,\delta_\epsilon]$$
and extended continuously by $1$.
The original metric $g$ can be written in the Fermi coordinates $\gamma_\epsilon$ as $g_t\oplus dt^2$. Now define on $U_\epsilon$ the following smooth metric $h_\epsilon$:
$$h_\epsilon(q)= \left\{ \begin{array}{rcl}
g_t(q) \oplus (\vartheta_\epsilon (q) dt)^2 & \text{ for $q\in \gamma_\epsilon(\partial U_\epsilon \times [0,\delta_\epsilon])$} \\
g(q) & \text{ for $q\in U_\epsilon\backslash \gamma_\epsilon(\partial U_\epsilon \times [0,\delta_\epsilon])$} 
\end{array}\right..$$ 
We just defined a compact manifold with boundary $(U_\epsilon, h_\epsilon)$. Let us state some useful properties of $h_\epsilon$. A first lemma controls the extrinsic curvature of the slices $\gamma_\epsilon(\partial U_\epsilon \times \{t\})$ for the new metric $h_\epsilon$.

\begin{lemme} \label{meancurv}
Let $(U_\epsilon,h_\epsilon)$ be defined as above. By abuse of notations, consider $\partial U_\epsilon \times [0,\delta_{\epsilon}]$ as a subset of $U_\epsilon$ via $\gamma_\epsilon$. Then for $t\in [0,\delta_{\epsilon}]$, the slices $\partial U_\epsilon \times \{t\}$ satisfy the following with respect to the new metric $h_\epsilon$:
\begin{enumerate}
\item they have non-zero mean curvature vector pointing in the direction of $-\frac{\partial}{\partial t}$,
\item their mean curvature goes uniformly to $0$ as $\epsilon$ converges to $0$,
\item their second fundamental form is bounded by a constant $C$ independent of $\epsilon$.
\end{enumerate}
\end{lemme}

\begin{proof} Choose $\frac{\partial}{\partial s}$ as the unit normal of the slices $\partial U_\epsilon \times \{t\}$ with respect to $h_\epsilon$, pointing in the direction of $\frac{\partial}{\partial t}$. Let $\mathbf{A}_{h_\epsilon}$ (resp. $\mathbf{A}_g$) be the second fundamental form of a slice with respect to $h_\epsilon$ (resp. $g$). Recall that $\partial U$ is a minimal hypersurface for $g$. Then the three items are proved readily by the fact that $\lim_{\epsilon\to 0}\delta_\epsilon =0$ and by observing that 
$$\frac{\partial}{\partial s} = \vartheta_\epsilon^{-1}\frac{\partial}{\partial t} \text{ and } \mathbf{A}_{h_\epsilon}(v,v) = \vartheta_\epsilon^{-1}\mathbf{A}_g(v,v)$$
for any $v$ belonging to the tangent  space of a slice $\partial U_\epsilon \times \{t\}$.

\end{proof}

\begin{remarque}\label{positivefoliation}
For later use, we point out that, as a consequence of the previous lemma, the union of $\Phi(\partial U \times \{t\})$ for $t\in[\epsilon,\hat{t}]$ and ${\gamma}_\epsilon( \partial {U}_\epsilon \times\{t\})$ for $t\in [0,\delta_{\epsilon}]$ form a continuous foliation of a neighborhood of $\partial U_\epsilon$ starting from $\Phi(\partial U \times \{\hat{t}\})$; each leaf is smooth and has a mean curvature vector which is non-zero and pointing towards $\partial U_\epsilon$ with respect to the metric $h_\epsilon$.

\end{remarque}

The next two lemmas give information on the asymptotic behavior of the metrics $h_\epsilon$. The following Lemma \ref{geomconv} shows that $(U_\epsilon,h_\epsilon)$ converges to the non-compact manifold with cylindrical ends $\mathcal{C}(U)$ and the convergence is smooth outside of the folding region where the curvature can be unbounded. In Lemma \ref{norm} we will describe how the folding region of $(U_\epsilon,h_\epsilon)$ is controlled in the $C^1$ topology by $g$.

Recall that a metric $g$ on a manifold $\Omega$ naturally determines a $C^k$ norm for smooth tensors on $\Omega$ and any nonnegative integer $k$. The notion of geometric convergence is defined for instance in \cite[Chapter 5, Section 1]{MorganTian}. We will say that a sequence $(U_k,g_k,x_k)$ converges geometrically to $(U_\infty,g_\infty, x_\infty)$ in the $C^k$ topology if $(U_k,g_k,x_k)$, $(U_\infty,g_\infty, x_\infty)$ satisfy the conditions in Definition 5.3 of \cite{MorganTian}, with the $C^\infty$-convergence on compact subsets replaced by $C^k$-convergence on compact subsets.

\begin{lemme} \label{geomconv}
The sequence of Riemannian manifolds $(U_\epsilon,h_\epsilon,q)$ converges geometrically to the non-compact manifold $(\mathcal{C}(U),h,q)$ in the $C^0$ topology. Moreover the geometric convergence is smooth outside of $\partial U\subset \mathcal{C}(U)$ in the following sense.
\begin{enumerate}
\item Let $q\in U\backslash \partial U$. For small $\epsilon$, we have $q\in U_\epsilon\backslash\gamma_\epsilon(\partial U_\epsilon\times [0,\delta_\epsilon))$. Then as $\epsilon\to 0$,
$$(U_\epsilon\backslash\gamma_\epsilon(\partial U_\epsilon\times [0,\delta_\epsilon]), h_\epsilon,q)$$ converges geometrically to $(U\backslash \partial U,g,q)$ in the $C^\infty$ topology.
\item Fix any connected component $C$ of $\partial U$; for $\epsilon$ small we can choose a component $C_\epsilon$ of $\partial U_\epsilon$ so that $C_\epsilon$ converges to $C$ as $\epsilon\to 0$ and the following holds. Let $\epsilon_k>0$ be a sequence converging to $0$. For all $k$ let $q_k \in \gamma_\epsilon(C_{\epsilon_k}\times [0,\delta_{\epsilon_k}))$ be a point at fixed distance $d'>0$ from $\gamma_\epsilon(C_{\epsilon_k}\times \{\delta_{\epsilon_k}\})$ for the metric $h_{\epsilon_k}$, $d'$ being independent of $k$. Then 
$$\big(\gamma_{\epsilon_k}(C_{\epsilon_k}\times [0,\delta_{\epsilon_k})),h_{\epsilon_k},q_k\big)$$
subsequently converges geometrically to $(C\times (-\infty,0),g_{prod}, q_\infty)$ in the $C^\infty$ topology, where $g_{prod}$ is the product of the restriction of $g$ to $C$ and the standard metric on $(-\infty,0)$, and $q_\infty$ is a point of $C\times (-\infty,0)$ at distance $d'$ from $C\times \{0\}$.

\end{enumerate}

\end{lemme}
\begin{proof}
These properties readily follow from the construction of $h_\epsilon$.
\end{proof}

We describe more precisely the folding region in $U_\epsilon$ by finding explicit local charts where the metrics $h_\epsilon$ converge in the $C^0$ topology while remaining bounded in the $C^1$ topology. The sign $\llcorner$ stands for the restriction of a metric $g$ to a submanifold.

\begin{lemme} \label{norm}
For any $\frak{d}_1\in(0,\hat{d})$, there exists $\eta>0$ 
such that for all $\epsilon>0$ small, there is an embbeding $\theta_\epsilon: \partial U\times [-\frak{d}_1,\frak{d}_1]\to U_\epsilon$ satisfying the following properties: 
\begin{enumerate}
\item $\theta_\epsilon(\partial  U\times \{0\}) = \gamma_\epsilon(\partial U_\epsilon\times\{\delta_\epsilon\})$ and 
$$\theta_\epsilon(\partial U\times [-\frak{d}_1,\frak{d}_1]) = \{q\in U_\epsilon; \quad d_{h_\epsilon}\big(q,\gamma_\epsilon(\partial U_\epsilon\times\{\delta_\epsilon\})\big)\leq \frak{d}_1\},$$
\item  $\|\theta_\epsilon^*h_\epsilon \|_{C^1(\partial U\times [-\frak{d}_1,\frak{d}_1])} \leq \eta$ where $\|.\|_{C^1(\partial U\times [-\frak{d}_1,\frak{d}_1])}$ is computed with the product metric $h':=g\llcorner \partial U  \oplus ds^2$,
\item the metrics $\theta_\epsilon^*h_\epsilon$ converge in the $C^0$ topology to a Lipschitz continuous metric which is smooth outside of $\partial U\times\{0\}\subset\partial U\times [-\frak{d}_1,\frak{d}_1]$ and 
$$\lim_{\epsilon \to 0}\|\theta_\epsilon^*h_\epsilon\llcorner (\partial U\times [0,\frak{d}_1]) - \gamma_0^*g\llcorner (\partial U\times [0,\frak{d}_1]) \|_{C^0({\partial U\times [0,\frak{d}_1]})}=0$$
where $\|.\|_{C^0(\partial U\times [0,\frak{d}_1])}$ is computed with $h'$.
\end{enumerate}

\end{lemme}

\begin{proof}
For $\epsilon>0$ small, it is possible to choose diffeomorphisms 
$$\phi_\epsilon : \partial U \to \partial U_\epsilon$$ such that 
$$\lim_{\epsilon\to 0}\phi_\epsilon^*(g\llcorner\partial U_\epsilon) =g\llcorner \partial U.$$
Let $\exp^\epsilon$ be the exponential map for the metric $h_\epsilon$. Then we define 
$$\theta_\epsilon:\partial U\times [-\frak{d}_1,\frak{d}_1]\to U_\epsilon$$
$$\theta_\epsilon(x,s) := \exp^\epsilon(\gamma_\epsilon(\phi_\epsilon(x),\delta_\epsilon),s\nu^\epsilon),$$
where $\nu^\epsilon$ is the unit normal of $\gamma_\epsilon(\partial U_\epsilon\times\{\delta_\epsilon\})$ for $h_\epsilon$ pointing away from $\partial U_\epsilon$. Note that $\nu^\epsilon$ is also a unit vector for $g$ since $h_\epsilon$ and $g$ coincide on $\gamma_\epsilon(\partial U_\epsilon\times\{\delta_\epsilon\})$. For $\frak{d}_1$ smaller than say, $\hat{d}/2$, this map is an embedding for all $\epsilon>0$ small.

Let us check that $\theta_\epsilon$ satisfies the conclusions of the lemma. The first item follows from the definition of $\theta_\epsilon$ in terms of $\exp^\epsilon$, the second item comes from the fact that the function $\vartheta_\epsilon$ appearing in the construction of $h_\epsilon$ is at least $1$. Finally the last bullet follows from the choice of $\phi_\epsilon$ and the fact that $g$, $h_\epsilon$ coincide on $U_\epsilon\backslash \big(\gamma_\epsilon(\partial U_\epsilon\times[0,\delta_\epsilon])\big)$

\end{proof}

\subsection{Definition of the min-max widths and the cylindrical Weyl law}
Let $(M,g)$ be a connected compact Riemannian manifold with boundary. The space of $k$-dimensional rectifiable mod $2$ flat chains with rectifiable boundary in a manifold $(M,g)$ is denoted by $\mathbf{I}_{k}(M;\mathbb{Z}_2)$, and $\mathbf{M}$ stands for the mass of a mod $2$ flat chain. We denote by $\mathcal{Z}_n(M,\partial M;\mathbb{Z}_2)$ the space of $T\in \mathbf{I}_{n}(M;\mathbb{Z}_2)$ in $(M,g)$ with $T=\partial U+T_1$ for some $(n+1)$-dimensional mod $2$ flat chain $U$ in $M$ and some $n$-dimensional mod $2$ flat chain $T_1$ with support in $\partial M$. 
As in \cite[Definition 1.20]{Alm1} (see also \cite[2.2]{LioMaNe}), we define from $\mathcal{Z}_n(M,\partial M;\mathbb{Z}_2)$ the space of relative cycles $\mathcal{Z}_{n,rel}(M,\partial M;\mathbb{Z}_2)$. The latter coincides with $\mathcal{Z}_{n}(M;\mathbb{Z}_2)$ when the boundary $\partial M$ is empty. The space of relative cycles is endowed with the flat topology and is weakly homotopically equivalent to $\mathbb{RP}^\infty$ (see \cite{Alm1} and Section 5 of \cite{MaNemultiplicityone} for the case $\partial M=\varnothing$). We denote by $\overline{\lambda}$ the generator of $H^1(\mathcal{Z}_{n,rel}(M,\partial M;\mathbb{Z}_2), \mathbb{Z}_2)=\mathbb{Z}_2$.

Let $X$ be a finite dimensional simplicial complex. A map $\Phi:X\rightarrow \mathcal{Z}_{n,rel}(M,\partial M;\mathbb{Z}_2)$ continuous in the flat topology is called a {\em  $p$-sweepout} if
$$
\Phi^*(\bar \lambda^p) \neq 0 \in H^p(X;\mathbb{Z}_2).
$$
By \cite[Theorem 2.10, Theorem 2.11]{LioMaNe} (based on \cite{MaNeWillmore}), it will make no difference if we restrict ourselves to maps $\Phi:X\rightarrow \mathcal{Z}_{n,rel}(M,\partial M;\mathbb{Z}_2)$ which are continuous in the much stronger mass topology.
We say 
that a $p$-sweepout $\Phi:X\rightarrow \mathcal{Z}_{n,rel}(M,\partial M;\mathbb{Z}_2)$ has {\em no concentration} of mass when 
$$\lim_{r\to 0} \sup\{\mathbf{M}(\Phi(x) \cap B_r(p)):x\in X, p\in M\}=0.$$
The set of all $p$-sweepouts $\Phi$ that have no concentration of mass is denoted by $\mathcal P_p=\mathcal P_p(M,g)$. Note that two maps in  $\mathcal P_p$ can have different domains.

In \cite{MaNeinfinity} and then \cite{LioMaNe}, the widths of $M$ were defined as
$$\omega_p(M,g)=\inf_{\Phi \in \mathcal P_p}\sup\{M(\Phi(x)): x\in {\rm dmn}(\Phi)\},$$
where ${\rm dmn}(\Phi)$ is the domain of $\Phi$ (see \cite{Gromovnonlinearspectra} and \cite{Guth} for previous works on the subject). As explained in \cite[Lemma 4.7]{MaNeinfinity} (at least for the case $\partial M = \varnothing$, but the following is true in general), the $p$-width can be expressed as the infimum of the widths of homotopy classes of discrete $p$-sweepouts (see \cite{MaNeinfinity} for definitions and notations):
\begin{equation} \label{detail}
\omega_p(M,g) = \inf_{\Pi\in \mathcal{D}_p} \mathbf{L}(\Pi).
\end{equation}
A rather technical point in the theory is that usually one proves the existence of smooth minimal hypersurfaces and Morse index bounds with Almgren-Pitts's theory by working with a fixed class of discrete sweepouts $\Pi$. However it is not clear that the infimum in (\ref{detail}) is realized by one particular $\Pi$.
Nevertheless, $\omega_p(M,g)$ is indeed achieved by an integral varifold whose support is a smooth minimal hypersurface of index bounded by $p$.
For instance when $\partial M=\varnothing$, one first reduces  to the use of discrete sweepouts with $k$-dimensional domains $X$, then one applies the index bound of \cite{MaNeindexbound} together with the compactness result of \cite{Sharp} (see \cite[Proposition 2.2]{IrieMaNe} for a detailed explanation). When $M$ has a non-trivial boundary $\partial M\neq\varnothing$, one can argue similarly. M. Li and X. Zhou \cite{LiZhoufree} extended the arguments of Almgren-Pitts to show that for each $\Pi$ as above, one can produce an integral varifold of mass $\mathbf{L}(\Pi)$ whose support is a smooth ``almost properly embedded" free boundary minimal hypersurface, see \cite[Section 2]{LiZhoufree} for definitions. Note that \cite[Theorem 4.21]{LiZhoufree} holds for discrete sweepouts with domains $X$ which are finite dimensional cubical complexes (see \cite{GLWZ21}). Then one can use the index bound of \cite{GLWZ21} together with \cite{Sharp} to conclude.

Another technical detail is the definitions of equivalence classes of relative cycles in \cite[Section 2]{LioMaNe} and \cite[Section 3]{LiZhoufree}: in \cite{LioMaNe} the authors use integral currents before defining the quotient space whereas in \cite{LiZhoufree}, the authors use currents which are only integer rectifiable. These formulations are equivalent by \cite[Lemma 3.8]{LiZhoufree}. They lead to the same notion of space of relative cycles $\mathcal{Z}_{n,rel}(M,\partial M;\mathbb{Z}_2)$ and the same functionals mass $\mathbf{M}$ and flat norm $\mathcal{F}$.

In our paper, we will need to consider the non-compact setting, so naturally we define the following.

\begin{definition} Let $(N,g)$ be a complete non-compact manifold. Let $K_1\subset K_2\subset... \subset  K_i\subset...$ be an exhaustion of $N$ by compact $(n+1)$-submanifolds with smooth boundary. The {\em $p$-width of $(N,g)$} is the number
$$\omega_p(N,g)=\lim_{i\to\infty} \omega_p(K_i,g) \in [0,\infty].$$
\end{definition}
For any two compact $(n+1)$-submanifolds with smooth boundary $A\subset B\subset N$, we have $ \omega_p(A,g)\leq  \omega_p(B,g)$; this follows from adapting the proof of \cite[Lemma 2.15, (1)]{LioMaNe}) to the case of general metrics. Therefore $\omega_p(N,g)$ is well-defined since $\omega_p(K_i,g)$ is a nondecreasing sequence of nonnegative numbers. Moreover it does not depend on the choice of the exhaustion $\{K_i\}$. 

Let $(\mathcal{C},h)$ be a complete manifold with cylindrical ends, i.e. outside of a compact subset, the manifold is isometric to $\Sigma \times [0,\infty)$ endowed with a product metric (here $\Sigma$ is a smooth $n$-dimensional manifold). The metric $h$ is allowed to be only Lipschitz continuous. We denote by $\mathcal{H}^n$ the Hausdorff $n$-dimensional volume. 

The following theorem is a key result in this paper: it states that the widths $\omega_p(\mathcal{C}) = \omega_p(\mathcal{C},h)$ increase with $p$ by a definite explicit amount (strict monotonicity), and that the asymptotic behavior of $\omega_p(\mathcal{C})$ is linear with an explicit leading term (cylindrical Weyl law).

\begin{theo} \label{asymptotic}
Let $(\mathcal{C},h)$ be an $(n+1)$-dimensional connected non-compact manifold with cylindrical ends, which is isometric to a product metric $(\Sigma \times [0,\infty),h_1\oplus dt^2)$ outside of a compact subset. Let $\Sigma_1,\dots,\Sigma_m$ be the connected components of $\Sigma$ and suppose that $\Sigma_1$ has the largest $n$-volume among these components:
$$\mathcal{H}^n(\Sigma_1)\geq \mathcal{H}^n(\Sigma_j) \text{ for all } j\in\{1,\dots,m\},$$
$\mathcal{H}^n$ being computed with $h_1$.
Then $\omega_p(\mathcal{C})$ is finite for all $p$ and the following holds:
\begin{enumerate}
\item for all $p\in\{1,2,3,\dots\}$,
$$
\omega_{p+1}(\mathcal{C})-\omega_p(\mathcal{C})\geq \mathcal{H}^n(\Sigma_1),
$$

\item moreover, there exists a constant $\hat{C}$ depending on $h$ such that for all $p\in\{1,2,3,\dots\}$:
$$ p\mathcal{H}^n(\Sigma_1) \leq \omega_p(\mathcal{C})\leq p\mathcal{H}^n(\Sigma_1) +\hat{C}p^{\frac{1}{n+1}}.$$
\end{enumerate}
\end{theo}

\begin{proof}
Suppose first that each $\omega_p(\mathcal{C})$ is finite. We will use a few times \cite[Lemma 2.15]{LioMaNe} which holds true more generally for Riemannian manifolds with smooth boundaries. By hypothesis, there is a compact subset $A\subset \mathcal{C}$ such that $(\mathcal{C}\backslash A,h)$ is isometric to $(\Sigma\times[0,\infty), h_1\times dt^2)$. We need the following basic fact on the first width of cylinders:
\begin{equation} \label{omega1}
\exists R_0>0, \forall R\geq R_0,\quad \omega_1(\Sigma_1\times[0,R], h_1\times dt^2)  = \mathcal{H}^n(\Sigma_1). 
\end{equation}
This follows by noticing two things. 
On the one hand, the hypersurfaces $\{\Sigma_1\times \{r\}\}_{r\in [0,R]}$ give an explicit sweepout in $\mathcal{P}_1$ for which each non-trivial slice has $n$-volume equal to $\mathcal{H}^n(\Sigma_1)$, so 
$$\omega_1(\Sigma_1\times[0,R], h_1\times dt^2)  \leq \mathcal{H}^n(\Sigma_1).$$
On the other hand by applying the min-max theory in the setting with boundary of M. Li and X. Zhou \cite{LiZhoufree}, we get a varifold $V$ such that 
\begin{itemize}
\item the mass of $V$ is arbitrarily close to $\omega_1(\Sigma_1\times[0,R], h_1\times dt^2) $ (see (\ref{detail} and the following comments), 
\item the support of $V$ is a smooth almost properly embedded minimal hypersurface (see \cite[Subsection 2.3]{LiZhoufree} for the definition of ``almost properly embedded"). 
\end{itemize}
By the maximum principle each connected component of $\spt(V)$ is either of the form $\Sigma_1\times \{r\}$ or it intersects all such slices. Hence by the monotonicity formula, if $R$ is large enough, 
$$\omega_1(\Sigma_1\times[0,R], h_1\times dt^2)  \geq \mathcal{H}^n(\Sigma_1) \text{ for $R$ large enough}$$ 
and so (\ref{omega1}) is proved.

Now we show that
\begin{equation}\label{bloh}
\omega_1(\mathcal{C})\geq \mathcal{H}^n(\Sigma_1) \text{ and } \omega_{p+1}(\mathcal{C})-\omega_{p}(\mathcal{C})\geq \mathcal{H}^n(\Sigma_1) \text{ for $p\in\{1,2,...\}$}.
\end{equation}
This will immediately yield item (1) and the first inequality in item (2). Let $E_0$ be a subset of $\mathcal{C}$ isometric to $(\Sigma_1\times [0,R_0], h_1\times dt^2)$. By \cite[Lemma 2.15 (1)]{LioMaNe} and by (\ref{omega1}), for all $R$ large enough and point $q\in \mathcal{C}$, we have $E_0\subset B(q,R)$:
$$\omega_1(\mathcal{C})\geq \omega_1(B(q,R),h)\geq \omega_1(E_0,h)= \mathcal{H}^1(\Sigma_1).$$
The second formula in (\ref{bloh}) follows from a Lusternick-Schnirelmann type argument in the setting of Almgren-Pitts theory already used by Marques-Neves \cite{MaNeinfinity} and then by Liokumovich-Marques-Neves \cite{LioMaNe}. Let $\mu>0$, fix a point $q\in \mathcal{C}$ and suppose that $R_1$ is big enough so that 
\begin{align} \label{larrge}
\begin{split} 
\forall R\geq R_1, \quad \omega_{p}(B(q,R),h) & > \omega_{p}(\mathcal{C})-\mu. 
\end{split}
\end{align}
Let $R>R_1$ be large enough so that $B(q,R)$ contains the disjoint union of $B(q,R_1)$ and a subset isometric to $E_0$. Given $\Phi\in \mathcal{P}_{p+1}(B(q,R))$ continuous in the mass topology, with $X=\dmn(\Phi)$, consider $U_1$ and $U_2$ the open subsets of $X$ given by \cite[Lemma 2.15 (2)]{LioMaNe} containing respectively the open subsets:
$$\{x\in X; \mathbf{M}(\Phi(x)\llcorner B(q,R_1)) < \omega_p(B(q,R_1),h) - \mu \}$$
$$ \text{and }\{x\in X; \mathbf{M}(\Phi(x)\llcorner E_0 ) < \omega_1(E_0,h) - \mu\}.$$
Arguing as in \cite[Theorem 3.1]{LioMaNe}, we obtain $X\neq U_1\cup U_2$. Let $x\in X\backslash (U_1\cup U_2)$. Then by (\ref{omega1}) and (\ref{larrge}):
\begin{align*}
\omega_{p+1}(\mathcal{C})
& \geq \omega_{p+1}(B(q,R),h)\\
& \geq \omega_p(B(q,R_1),h)+\omega_1(E_0,h) -2\mu\\
& \geq \omega_{p}(\mathcal{C})+\mathcal{H}^n(\Sigma_1) -3\mu.
\end{align*}
Using \cite[Corollary 2.13]{LioMaNe}, and making $\mu$ go to $0$ we obtain
$$\omega_{p+1}(\mathcal{C}) \geq \omega_{p}(\mathcal{C}) + \mathcal{H}^n(\Sigma_1).$$

We need to show that the widths $\omega_p(\mathcal{C})$ are finite and satisfy the second inequality in item (2). One way would be to construct an explicit $p$-sweepout in $\mathcal{P}_p$ as in \cite[Theorem 5.1]{MaNeinfinity} using the bend-and-cancel argument of Guth \cite{Guth}. Instead we use the gluing technique of Liokumovich-Marques-Neves \cite{LioMaNe} which enables to combine the $p$-sweepouts over the same domain $X$ of several compact submanifolds $A_1,\dots,A_k$ with disjoint interiors 
into one $p$-sweepout over $X$ of the union $A_1\cup\dots\cup A_k$. By assumption there is a connected compact submanifold with boundary $A$ such that $(\mathcal{C}\backslash A,g)$ is isometric to $(\Sigma\times[0,\infty),h_1\times dt^2)$. We view $\Sigma\times [0,\infty)$ as an $(n+1)$-submanifold of $\mathcal{C}$. Let $L>0$ and define $B_L:=\Sigma\times [0,L]$. The boundaries of the two submanifolds $A$ and $B_L$ intersect along the closed hypersurface $\Sigma\times\{0\}$. Fix $p$. By the sublinear bound on the $p$-widths \cite[Theorem 5.1]{MaNeinfinity} (see also \cite[Section 8]{Gromovwaist} \cite[Theorem 1]{Guth}) which holds for compact manifolds with boundary with a Lipschitz metric, there is a $p$-sweepout $\Phi:\mathbb{RP}^p\to \mathcal{Z}_{n,rel}(A,\partial A;\mathbb{Z}_2)$ in $\mathcal{P}_p$ and there is a constant ${C}>0$ depending on $(A,h)$ but independent of $p$ such that 
\begin{equation} \label{massbound}
\sup_{x\in \mathbb{RP}^p}\mathbf{M}(\Phi(x))\leq {C} p ^{\frac{1}{n+1}}.
\end{equation}
Recall that $\Sigma_1,\dots,\Sigma_m$ are the connected components of $\Sigma$. Let $f:B_L\to \mathbb{R}$ be the Morse function defined by $f(x,t):=(j-1)L+t$ if $(x,t)\in \Sigma_j\times[0,L]$. As in \cite[Theorem 5.1]{MaNeinfinity}, $f$ determines a $p$-sweepout in $\mathcal{P}_p$ as follows.  Define
$$\hat{\Psi} : \{a\in\mathbb{R}^{p+1}; |a|=1\} \to \mathcal{Z}_{n,rel}(B_L,\partial B_L;\mathbb{Z}_2)$$
$$\hat{\Psi}(a_0,\dots,a_p) = \partial \{x\in B_L ; \sum_{i=0}^pa_i f(x)^i<0\}.$$
Since $\hat{\Psi}(a)=\hat{\Psi}(-a)$, $\hat{\Psi}$ induces a map $\Psi:\mathbb{RP}^{p}\to \mathcal{Z}_{n,rel}(B_L,\partial B_L;\mathbb{Z}_2)$ which is a $p$-sweepout in $\mathcal{P}_p$. Note that since we assume $\mathcal{H}^n(\Sigma_1)\geq \mathcal{H}^n(\Sigma_j) \text{ for all } j\in\{1,\dots,m\}$ where the $n$-volume is computed with $h_1$, we have
 \begin{equation} \label{massbound1}
\sup_{x\in \mathbb{RP}^p}\mathbf{M}(\Psi(x))\leq p\mathcal{H}^n(\Sigma_1).
 \end{equation}
Note that $B_L= \bigcup_{j=1}^m (\Sigma_j \times[0,L])$ and that each $\Sigma_j \times[0,L]$ is connected. In \cite[Section 4]{LioMaNe}, it is shown that for all $i\in\{1,\dots,m\}$,
\begin{align*}
SX_{0}: & =\{(x,Z) ; \quad x\in \mathbb{RP}^p, \Phi(x)-\partial Z \in \mathbf{I}_n(\partial A;\mathbb{Z}_2)\} \\
& \subset \mathbb{RP}^p\times \mathbf{I}_{n+1}(A;\mathbb{Z}_2),
\end{align*}
\begin{align*}
SX_i:& =\{(x,Z) ; \quad x\in \mathbb{RP}^p, \Psi(x)\llcorner (\Sigma_i \times[0,L])-\partial Z \in \mathbf{I}_n(\Sigma_i \times\{0,L\};\mathbb{Z}_2)\} \\
& \subset \mathbb{RP}^p\times \mathbf{I}_{n+1}(\Sigma_i \times[0,L];\mathbb{Z}_2)
\end{align*}
are isomorphic $2$-covers of $\mathbb{RP}^p$ (they are actually isomorphic to the sphere $S^p$). Let $F_i:S^p\to SX_i$ ($i=0,\dots,m$) be the corresponding isomorphism. Let $\Xi_i:SX_i\to \mathbf{I}_{n+1}(A\cup B_L;\mathbb{Z}_2)$ ($i=0,\dots,m$) be the natural projection. Set
$$\hat{\varkappa}_L: S^p \to \mathbf{I}_{n}(A\cup B_L;\mathbb{Z}_2)
$$
$$\hat{\varkappa}_L(y) = \sum_{i=0}^{m}\partial(\Xi_i\circ F_i(y)). $$
Since for any $y\in S^p$, $\hat{\varkappa}_L(y)=\hat{\varkappa}_L(-y)+\partial (A\cup B_L)$ in $\mathbf{I}_{n}(A\cup B_L;\mathbb{Z}_2)$, $\hat{\varkappa}_L$ induces a map 
$$\varkappa_L : \mathbb{RP}^p\to \mathcal{Z}_{n,rel}(A\cup B_L,\partial( A\cup B_L);\mathbb{Z}_2).$$
By (\ref{massbound}) and (\ref{massbound1}), 
$$\sup_{x\in \mathbb{RP}^p} \mathbf{M}(\varkappa_L(x)) \leq p\mathcal{H}^n(\Sigma_1)+{C}p^{\frac{1}{n+1}} + \mathcal{H}^n(\Sigma).$$
The map $\varkappa_L$ is a $p$-sweepout without concentration of mass (see \cite{LioMaNe}) thus for $\hat{C}=C+\mathcal{H}^n(\Sigma)$,  
$$\omega_p(\mathcal{C}) = \lim_{L\to \infty}\omega_p(A \cup B_L) \leq  \sup_{x\in \mathbb{RP}^p} \mathbf{M}(\varkappa_L(x)) \leq p\mathcal{H}^n(\Sigma_1) +\hat{C}p^{\frac{1}{n+1}} $$
and so the widths $\omega_p(\mathcal{C})$ are finite and satisfy the second inequality of item (2). The theorem is proved.

\end{proof}

\subsection{Confined min-max closed minimal hypersurfaces} \label{confined}
Let $(U,g)$ and $(\mathcal{C}(U),h)$ be as in Subsection \ref{construction}. In the following, we will prove that the widths $\omega_p(\mathcal{C}(U))$ are realized as the $n$-volume (with multiplicities) of smooth closed minimal hypersurfaces confined in the non-cylindrical part, i.e. the interior of $U$. The proof will use the approximations $(U_\epsilon,h_\epsilon)$ previously constructed.

\begin{theo} \label{minmax}
Let $(\mathcal{C}(U),h)$ be as in Subsection \ref{construction}. For all $p\in\{1,2,3,\dots\}$, there exist disjoint smooth closed connected embedded minimal hypersurfaces $\Gamma_1,\dots,\Gamma_N$ contained in $U\backslash \partial U$ and positive integers $m_1,\dots,m_N$ such that
$$\omega_p(\mathcal{C}(U)) = \sum_{j=1}^N m_j \mathcal{H}^n(\Gamma_j).$$
Besides, if $\Gamma_j$ is one-sided then the corresponding multiplicity $m_j$ is even.
\end{theo}

\begin{proof}

\textbf{Step 1.}
Consider the compact smooth approximations $(U_\epsilon,h_\epsilon)$ constructed in Subsection \ref{construction}. We fix $p$ and apply the Almgren-Pitts min-max theorem for $p$-sweepouts with relative cycles in the setting with boundary, developed by M. Li and X. Zhou \cite{LiZhoufree} (their results hold for more general sweepouts as defined in \cite{MaNeinfinity, LioMaNe}). We obtain a varifold $V_\epsilon$ of total mass $\epsilon$-close to $\omega_p(U_\epsilon,h_\epsilon)$ by (\ref{detail}) and the following remarks. The support of $V_\epsilon$ is a smooth compact almost properly embedded free boundary minimal hypersurface $S_\epsilon \subset (U_\epsilon,h_\epsilon)$ (see \cite[Subsection 2.3]{LiZhoufree} for the definition of ``almost properly embedded"). Since the boundary $\partial U_\epsilon$ is strictly mean concave with respect to the unit normal pointing outside of $U_\epsilon$, no component of the hypersurface $S_\epsilon$ is contained in the boundary $\partial U_\epsilon$. Moreover by 
the maximum principle and Remark \ref{positivefoliation}, if a connected component of $S_\epsilon$ intersects $\gamma_\epsilon(\partial U_\epsilon \times [0,\delta_\epsilon])$ then it also has to intersect $\Phi(\partial U \times \{\hat{t}\})$ (see Subsection \ref{construction} for notations). Since by Lemma \ref{geomconv} (2), the regions $\gamma_\epsilon(\partial U_\epsilon \times [0,\delta_\epsilon])$ look more and more cylindrical, the monotonicity formula indicates that there is a radius $\tilde{R}>0$ and a point $q\in U\backslash \partial U$ such that for all $\epsilon>0$ small enough, the varifold $V_\epsilon$ is supported in the ball $B_{h_\epsilon}(q,\tilde{R})$. In particular for $\epsilon$ small, all the smooth minimal hypersurfaces produced by the min-max theory are closed minimal hypersurface embedded in the interior of $U_\epsilon$. Besides the one-sided components have even multiplicities: this follows from the Multiplicity One theorem \cite{Zhoumultiplicityone,SWZ20}, and an elementary perturbation and compactness argument.
Note that for any $\epsilon$ fixed and small, these facts are true for small perturbations of the metric $h_\epsilon$ and so \cite[Theorem 1.2]{MaNeindexbound} is true for $(U_\epsilon,h_\epsilon)$ (see the comment after (\ref{detail})).



\vspace{1em}
\textbf{Step 2.}
The next step is to take a limit as a sequence $\epsilon_k\to 0$ and argue that $S_{\epsilon_k}$ converges to a smooth closed minimal hypersurface contained in the interior of $U$. Note that $\omega_p(U_\epsilon,h_\epsilon)$ converges to $\omega_p(\mathcal{C}(U),h)$. Thus for a sequence $\epsilon_k \to 0$, the varifolds $V_{\epsilon_k}$ produced by min-max subsequently converge in the varifold sense to a varifold $V_\infty$ in $\mathcal{C}(U)$ of total mass $\omega_p(\mathcal{C}(U),h)$, whose support is denoted by $S_\infty$. Moreover $S_\infty$ is contained in the ball $B_h(q,2\tilde{R})$. Even though the metrics $h_\epsilon$ converge to $h$ only in the $C^0$ topology and $h$ is not smooth, the varifold convergence makes sense as usual by using Lemma \ref{geomconv} and Lemma \ref{norm}.


%
By the index bound of Marques-Neves \cite[Theorem 1.2]{MaNeindexbound} which still holds in our situation with boundary, the minimal hypersurfaces $S_\epsilon$ can be chosen to have index bounded by $p$ when $\epsilon$ is small. Thus by Sharp's compactness result \cite{Sharp}, for a well chosen sequence $\epsilon_k \to 0$, the restriction of the support $S_\infty$ of $V_\infty$ to $\mathcal{C}(U) \backslash \partial U$ is a smooth minimal hypersurface, since the metrics $h_\epsilon$ converge smoothly in this region by Lemma \ref{geomconv}. We first observe that by the maximum principle any component of $S_\infty$ in the cylindrical part $\mathcal{C}(U) \backslash U$ would be a slice isometric to $(\partial U,g)$. But as we explained two paragraphs before, any component of $S_{\epsilon_k}$ intersects $\Phi(\partial U \times \hat{t})$ so it shows that the support $S_\infty$ is contained in the compact set $(U,g)$. Secondly, we wish to 
prove that $V_\infty$ is $g$-stationary. Indeed supposing this is verified, then applying Lemma \ref{norm} and Proposition \ref{Allard}, we obtain that $V_\infty$ is actually a $g$-stationary integral varifold. By the monotonicity formula, no component of $S_\infty$ is contained in $\partial U$. Hence the maximum principle proved by White \cite{WhiteMP} (based on \cite{SolomonWhiteMP}) implies that $S_\infty \cap \partial U = \varnothing$. Therefore the support $S_\infty$ of $V_\infty$ is a closed minimal hypersurface smoothly embedded inside $\interior(U)$. By the description of convergence of finite index minimal hypersurfaces \cite{Sharp}, the one-sided components of $S_\infty$ have even multiplicities. This would complete the proof.

\vspace{1em}
\textbf{Step 3.}
In the remaining steps of the proof, we show that $V_\infty$ is $g$-stationary in $U$. It is non-trivial only because the limit metric $h$ is not smooth in general. 

In Step 3, we describe what $V_\epsilon$ looks like around the boundary $\partial U_\epsilon$.
For $\epsilon\geq0$, we will denote by $\divergence^\epsilon$ the divergence computed in the metric $h_\epsilon$ (by convention $h_0=g$). Let $X$ be any vector field smooth up to the boundary $U$. Our goal is to prove that the first variation along $X$ vanishes:
\begin{equation} \label{stat}
\delta V_\infty(X)= \int \divergence^0_S X(x) dV_\infty(x,S) = 0.
\end{equation}
Recall that for all $\epsilon\geq 0$ small, the map
$$\gamma_\epsilon : \partial U_\epsilon \times [0,\hat{d}] \to U$$
$$\gamma_\epsilon(x,t) = \exp(x, t\nu)$$
is a diffeomorphism onto its image, where $\hat{d}>0$ is independent of $\epsilon$. Here $\exp$ is the exponential map for $g$, and $\nu$ is the inward unit normal of $\partial U_\epsilon$ for $g$. By abuse of notations, we define the following vector fields:
$$\frac{\partial}{\partial t}:=(\gamma_\epsilon)_*(\frac{\partial}{\partial t}),$$
$$\frac{\partial}{\partial s}:=\frac{(\gamma_\epsilon)_*(\frac{\partial}{\partial t})}{\|(\gamma_\epsilon)_*(\frac{\partial }{\partial t})\|_{h_\epsilon}},$$
so that $\frac{\partial}{\partial t}$ (resp. $\frac{\partial}{\partial s}$) is a unit vector field for $g$ (resp. $h_\epsilon$). Let $s$ be the coordinate obtained by integrating $ \frac{\partial}{\partial s}$ from $\gamma_\epsilon(\partial U_\epsilon\times \{\delta_\epsilon\})$, so that $\{s=0\}=\gamma_\epsilon(\partial U_\epsilon\times \{\delta_\epsilon\})$ and $s<0$ on $\gamma_\epsilon(\partial U_\epsilon\times [0,\delta_\epsilon))$.
Because the support of $V_\infty$ restricted to $\interior(U)$ is a smooth minimal hypersurface, we can assume that the vector field $X$ is supported inside $\gamma_0(\partial U\times [0,\hat{d}/2])$. Thus for all $\epsilon$ small enough, the vector field $X$ restricted to $U_\epsilon $ can be decomposed into two components
$$X=X^\epsilon_{\perp}+X^\epsilon_{\parallel} $$
where $X^\epsilon_{\perp}$ is orthogonal to $\frac{\partial}{\partial t}$ (for either $g$ or $h_\epsilon$) and $X^\epsilon_{\parallel}$ is a multiple of $\frac{\partial}{\partial t}$. 


Let $(y,S)$ be a point of the Grassmannian bundle of $U$ and let $(e_1,\dots,e_n)$ be an $h_\epsilon$-orthonormal basis of $S$ so that $e_1,\dots,e_{n-1}$ are $h_\epsilon$-orthogonal to $\frac{\partial}{\partial s}$. Let $e^*_n$ be a unit vector such that $(e_1,\dots,e^*_n)$ is an $h_\epsilon$-orthonormal basis of the $n$-plane $h_\epsilon$-orthogonal to $\frac{\partial}{\partial s}$ at $y$.
If $y=\gamma_\epsilon(x,t)\in\gamma_\epsilon(\partial U_\epsilon\times[0,\hat{d}])$, let $H^\epsilon$ (resp. $\mathbf{A}^\epsilon$) denote the mean curvature (resp. second fundamental form) at $y$ of $\gamma_\epsilon(\partial U_\epsilon\times\{t\})$, with respect to $h_\epsilon$ and the unit normal vector $ \frac{\partial}{\partial s}$. Let $C>0$ be a constant larger than the norm of $\mathbf{A}^\epsilon$: in view of Lemma \ref{meancurv} (3), $C$ can be chosen independently of $y$ and $\epsilon$. 

Let $a_1<b_1<b_2<a_2$ be four numbers in $(-\infty,\hat{d})$. Consider a smooth nondecreasing function $\beta:\mathbb{R} \to [0,1]$ such that 
\begin{itemize}
\item $\beta(s)\equiv0$ (resp. $1$) when $s\leq a_1$ (resp. $s\geq a_2$),
\item on $[b_1,b_2]$, $\frac{\partial \beta}{\partial s} \geq \frac{1}{2(a_2-a_1)}$.
\end{itemize} 

Define the following vector field on $U_\epsilon$:
$$Y^\epsilon:= (1-\beta(s)) \exp(-Cs) \frac{\partial }{\partial s},$$

For all $\epsilon>0$, we compute similarly to the proof of \cite[Lemma 2.2]{MaNe} the divergence of $Y^\epsilon$ 
(with respect to the metric $h_\epsilon$) at $(y,S)$:
\begin{align} \label{divY}
\begin{split}
\divergence^\epsilon_S Y^\epsilon  = & \frac{\partial (1-\beta(s))}{\partial s} \exp(-Cs) |h_\epsilon(e_n,\frac{\partial}{\partial s})|^2  \\ 
& -C(1-\beta(s))\exp(-Cs)|h_\epsilon(e_n,\frac{\partial}{\partial s})|^2 \\ 
& +(1-\beta(s))\exp(-Cs) H^\epsilon \\ 
&- (1-\beta(s)) \exp(-Cs) \mathbf{A}^\epsilon(e^*_n,e^*_n) |h_\epsilon(e_n,\frac{\partial}{\partial s})|^2 \\
= &\exp(-Cs) |h_\epsilon(e_n,\frac{\partial}{\partial s})|^2  \Big( \frac{\partial (1-\beta(s))}{\partial s} - C(1-\beta(s)) - (1-\beta(s)) \mathbf{A}^\epsilon(e^*_n,e^*_n) \Big) \\ 
& +(1-\beta(s))\exp(-Cs))H^\epsilon  \\
\leq & -\exp(-Cs)\frac{\partial \beta(s)}{\partial s} |h_\epsilon(e_n,\frac{\partial}{\partial s})|^2  +(1-\beta(s))\exp(-Cs))H^\epsilon.
\end{split}
\end{align}
For the inequality, we used that $C$ is larger than the norm of the second fundamental forms. 
Note that since the limit varifold $V_\infty$ has support in $U$, by the monotonicity formula and Lemma \ref{geomconv} (2), the coordinate $s$ stays bounded independently of $\epsilon$ on $\spt(V_\epsilon)$ when defined. Recall that by Lemma \ref{meancurv} (2), the term $(1-\beta(s))H^\epsilon$ goes to $0$ as $a_2\to 0$, uniformly in $\epsilon$. Since the varifold $V_\epsilon$ is $h_\epsilon$-stationary, for all $\epsilon>0$ small:
$$\delta V_\epsilon(Y^\epsilon) =\int \divergence^\epsilon Y^\epsilon dV_\epsilon=0$$
computed with $h_\epsilon$, for any $\epsilon$ small and any choice of $a_1<b_1<a_2<b_2$ in the definition of $\beta$. Recall also that the mass of $V_\epsilon$ is bounded uniformly in $\epsilon$ and that its support stays uniformly far from the boundary of $(U_\epsilon,h_\epsilon)$ for $\epsilon$ small.

Collecting these facts together, the previous computation implies the following. If we fix $b_1<0$ and let $a_2$ and $b_2>0$ tend to $0$, inequality (\ref{divY}) and the fact that $\delta V_\epsilon(Y^\epsilon) =0$ show that with the previous notation for $e_n$:
\begin{equation} \label{consequence1}
\begin{split}
& \forall \hat{a}>0, \forall \hat{b}>0, \exists \hat{r}>0, \exists \epsilon_0, \forall \epsilon\in(0,\epsilon_0], \quad \\ & \int_{\gamma_\epsilon(\partial U_\epsilon\times[0,\delta_\epsilon+\hat{r}]) \times \mathbf{G}(n+1,n)} \chi_{\{|h_\epsilon(e_n, \frac{\partial}{\partial s})|>\hat{b}\}} dV_\epsilon(x,S) <\hat{a},
\end{split}
\end{equation}
where $\chi_{\{|h_\epsilon(e_n, \frac{\partial}{\partial s})|>\hat{b}\}}$ is the characteristic function of the set of $(x,S)$ such that $|h_\epsilon(e_n, \frac{\partial}{\partial s})|>\hat{b}$. In particular by taking $\epsilon \to 0$ in the previous expression,
\begin{equation} \label{consequence1'}
V_\infty\llcorner \{(x,S); x\in \partial U, S\neq T_x\partial U\} =0.
\end{equation}

\vspace{1em}
\textbf{Step 4.}

In the final step, we explain how to deduce  (\ref{stat}) from the previous estimates. Let $\{\hat{r}_k\}$ be a sequence of positive numbers converging to $0$. Consider
$$A_k:=\gamma_\epsilon([0,\delta_\epsilon+\hat{r}_k]),$$
$$B_k:=U_\epsilon \backslash A_k.$$
Then, by taking a subsequence of $\{\epsilon_k\}$ and renumbering if necessary, we can assume that there are two varifolds $V'_\infty$, $V''_\infty$ in $U$ so that as $k\to \infty$, the following convergences in the varifold sense take place:
$$V_k:=V_{\epsilon_k}\rightharpoonup V_\infty $$
$$V'_k:=V_{\epsilon_k}\llcorner(A_k\times \mathbf{G}(n+1,n)) \rightharpoonup V'_\infty,$$
$$V''_k:=V_{\epsilon_k}\llcorner(B_k\times \mathbf{G}(n+1,n)) \rightharpoonup V''_\infty.$$
Recall that we decomposed $X=X^\epsilon_{\perp}+X^\epsilon_{\parallel} $. We will show first  that 
\begin{equation} \label{separately1}
\lim_{k\to \infty} \int \divergence^{\epsilon_k} X^{\epsilon_k}_{\perp} dV_k= \int \divergence^0 X^0_{\perp} dV_\infty
\end{equation}
and this will imply 
$$\int \divergence^0 X^0_{\perp} dV_\infty=0$$
since the left-hand side of (\ref{separately1}) is zero by $h_\epsilon$-stationarity of $V_k$. Then we will show less directly that 
\begin{equation} \label{separately2}
 \int \divergence^0 X^0_{\parallel} dV_\infty=0.
\end{equation}
Together, (\ref{separately1}) and (\ref{separately2}) would imply (\ref{stat}). In both cases, the decomposition $V_k=V'_k+V''_k$ is useful.

\vspace{1em}
\textbf{Step 4a.}
Let us  start with $X^0_{\perp}$.
We record the fact that, since on $B_k $ the metric $h_\epsilon$ coincides with the original metric $g$, for any smooth vector field $\tilde{X}$ on $M$: 
\begin{equation} \label{limit1}
\lim_{k\to \infty} \int \divergence^{\epsilon_k} \tilde{X} dV''_k=\lim_{k\to \infty} \int \divergence^{0} \tilde{X} dV''_k= \int \divergence^0 \tilde{X} dV''_\infty.
\end{equation}
In particular we already obtain 
$$\lim_{k\to \infty} \int \divergence^{\epsilon_k} X^{\epsilon_k}_{\perp} dV''_k=\lim_{k\to \infty} \int \divergence^{\epsilon_k} X^0_\perp dV''_k= \int \divergence^0X^0_\perp dV''_\infty$$
since $X^\epsilon_{\perp}$ converges smoothly to $X^0_{\perp}$ on $U$. To prove the analogous convergence for $V'_k$, we evaluate the divergence of $X^\epsilon_{\perp} $ in $h_\epsilon$ for $\epsilon>0$ as follows. 
Let $\nabla^\epsilon$ be the Levi-Civita connection of $h_\epsilon$. By definition of $h_\epsilon$, for any $\epsilon>0$ the restriction of $h_\epsilon$ and $g$ to a slice $\gamma_\epsilon(\partial U_\epsilon\times\{t\})$ coincide. Thus for any vector $e' $ in the tangent bundle of a slice $\gamma_\epsilon(\partial U_\epsilon\times\{t\})$, we have
\begin{equation} \label{bydef}
h_\epsilon(\nabla^\epsilon_{e'}X^\epsilon_{\perp} , e') = g(\nabla^0_{e'} X^\epsilon_{\perp} , e').
\end{equation}
If $(x,S)$, let $e_1,\dots,e_n,e^*_n$ be defined as before and let $S_\perp$ denote the $n$-plane at $x$ orthogonal to $\frac{\partial}{\partial s}$. For all $(x,S)$ such that $x\in A_k$, we have:
\begin{align}\label{long}
\begin{split}
\divergence^\epsilon_S X^\epsilon_{\perp}  
= & \sum_{i=1}^n h_\epsilon(\nabla^\epsilon_{e_i}X^\epsilon_{\perp},e_i) \\
=& \sum_{i=1}^{n-1} h_\epsilon(\nabla^\epsilon_{e_i}X^\epsilon_{\perp},e_i) +h_\epsilon(\nabla^\epsilon_{e^*_n}X^\epsilon_{\perp},e^*_n) +\big(h_\epsilon(\nabla^\epsilon_{e_n}X^\epsilon_{\perp},e_n)-h_\epsilon(\nabla^\epsilon_{e^*_n}X^\epsilon_{\perp},e^*_n)\big)\\
=& \sum_{i=1}^{n-1} g(\nabla^0_{e_i}X^\epsilon_{\perp},e_i) +g(\nabla^0_{e^*_n}X^\epsilon_{\perp},e^*_n) +\big(h_\epsilon(\nabla^\epsilon_{e_n}X^\epsilon_{\perp},e_n)-h_\epsilon(\nabla^\epsilon_{e^*_n}X^\epsilon_{\perp},e^*_n)\big)\\
=& \divergence^0_{S_\perp} X^\epsilon_{\perp}+ \Upsilon(\epsilon,x,S,X), 
\end{split}
\end{align}
where $\Upsilon(.)$ is a real function converging uniformly to $0$ as $|h_\epsilon(e_n,\frac{\partial}{\partial s})| \to 0$. 
The third equality comes from (\ref{bydef}). For the last equality we used the fact that $\|\nabla^\epsilon X^\epsilon_{\perp}\|_{h_\epsilon}$ is bounded uniformly in $\epsilon>0$: indeed recall from Subsection \ref{construction} that $h_\epsilon$ is obtained by stretching $g$ only in the $\frac{\partial}{\partial t}$ direction.

Now by the fact that $S_\infty=\spt(V_\infty)\subset U$ and (\ref{consequence1'}), we see that $V'_\infty$ is entirely supported on $\{(x,S); x\in \partial U, S=T_x\partial U\}$ as a Radon measure. For any $\hat{a}>0$, let  $\epsilon_k>0$ and $\hat{b}>0$ be small enough so that (\ref{consequence1}) is true with $\hat{r}=\hat{r}_k$ and $|\Upsilon(b)|\leq \hat{a}$. Then thanks to (\ref{long}) we can write that for any $k$ large enough:
\begin{align*}
& \big|\int \divergence^{\epsilon_k}_S X^{\epsilon_k}_{\perp} dV'_k(x,S))  - \int \divergence^0 X^0_\perp dV'_\infty\big| \\
 &=   \big|\int \chi_{\{|h_{\epsilon_k}(e_n,\frac{\partial}{\partial s})|\leq \hat{b}\}}\divergence^{\epsilon_k}_S X^{\epsilon_k}_{\perp} dV'_k(x,S)+  \int \chi_{\{|h_{\epsilon_k}(e_n,\frac{\partial}{\partial s})|> \hat{b}\}}\divergence^{\epsilon_k} X^{\epsilon_k}_{\perp} dV'_k- \int \divergence^0 X^0_{\perp} dV'_\infty\big|\\
  & \leq   \big|\int \chi_{\{|h_{\epsilon_k}(e_n,\frac{\partial}{\partial s})|\leq \hat{b}\}}\divergence^{\epsilon_k}_S X^{\epsilon_k}_{\perp} dV'_k(x,S)- \int \divergence^0 X^0_{\perp} dV'_\infty\big| + C\hat{a}\\
& \leq     \big|\int \chi_{\{|h_{\epsilon_k}(e_n,\frac{\partial}{\partial s})|\leq \hat{b}\}}\divergence^0_{S_\perp} X^{\epsilon_k}_{\perp} dV'_k(x,S)- \int \divergence^0 X^0_{\perp} dV'_\infty\big| +2 C\hat{a}\\
 &\leq \big|\int \divergence^0_{S_\perp} X^{\epsilon_k}_{\perp} dV'_k(x,S)- \int \divergence^0_{S_\perp} X^0_{\perp} dV'_\infty(x,S)\big| + 3C\hat{a}\\
&\leq 4C\hat{a}.
\end{align*}
Here $C$ is a constant depending on $X$ but neither on ${\epsilon_k}$ nor on $\hat{a},\hat{b}$. The last inequality can be seen in the chart $\theta_\epsilon$ of Lemma \ref{norm}: in these coordinates, it is clear that the divergence term $\divergence^0_{S_\perp} X^{\epsilon_k}_{\perp}$ would be unchanged if computed with the original metric $g$. Hence we just finished the proof  of (\ref{separately1}).

\vspace{1em}
\textbf{Step 4b.}
Next we study $X^0_\parallel$. Write $X^\epsilon_\parallel = \varphi_\epsilon \frac{\partial}{\partial t}$ where $\varphi_\epsilon$ is a smooth function on $U_\epsilon$ bounded independently of $\epsilon$. Then we define the vector field
$$Z^\epsilon:=\varphi_\epsilon \frac{\partial}{\partial s},$$
that is, we rescale $X^\epsilon_\parallel$ by $\vartheta_\epsilon^{-1}$ so that $Z^\epsilon$ is bounded for $h_\epsilon$. This particular definition of $Z^\epsilon$ is not essential, what counts is that 
\begin{itemize}
\item $Z^\epsilon$ is a multiple of $\frac{\partial}{\partial s}$,
\item $Z^\epsilon=X^\epsilon_\parallel$ on $U_\epsilon\backslash\big(\gamma_\epsilon(\partial U_\epsilon \times [0,\delta_\epsilon])\big)$ and
\item $\|\nabla^\epsilon Z^\epsilon\|_{h_\epsilon}$ is bounded uniformly in $\epsilon>0$.
\end{itemize}
Again to see why the third bullet is true, recall that $h_\epsilon$ is obtained by stretching $g$ in the $\frac{\partial}{\partial t}$ direction:
in particular  the differential of $\varphi_\epsilon$ with respect to $h_\epsilon$ is uniformly bounded and so is the covariant derivative of $Z^\epsilon$. Let $H^\epsilon$ be defined as before. For all $(x,S)$ such that $x\in A_k$,
\begin{align}\label{long'}
\begin{split}
\divergence^\epsilon_S Z^\epsilon
= & \sum_{i=1}^n h_\epsilon(\nabla^\epsilon_{e_i}Z^\epsilon,e_i) \\
=& \sum_{i=1}^{n-1} h_\epsilon(\nabla^\epsilon_{e_i}Z^\epsilon,e_i) +h_\epsilon(\nabla^\epsilon_{e^*_n}Z^\epsilon,e^*_n) +\big(h_\epsilon(\nabla^\epsilon_{e_n}Z^\epsilon,e_n)-h_\epsilon(\nabla^\epsilon_{e^*_n}Z^\epsilon,e^*_n)\big)\\
=& \sum_{i=1}^{n-1} h_\epsilon(\nabla^\epsilon_{e_i}\varphi_\epsilon \frac{\partial}{\partial s},e_i) +h_\epsilon(\nabla^\epsilon_{e^*_n}\varphi_\epsilon \frac{\partial}{\partial s},e^*_n) +\big(h_\epsilon(\nabla^\epsilon_{e_n}Z^\epsilon,e_n)-h_\epsilon(\nabla^\epsilon_{e^*_n}Z^\epsilon,e^*_n)\big)\\
=& \varphi_\epsilon H^\epsilon + \big(h_\epsilon(\nabla^\epsilon_{e_n}Z^\epsilon,e_n)-h_\epsilon(\nabla^\epsilon_{e^*_n}Z^\epsilon,e^*_n)\big)\\
=& \Upsilon'(\epsilon,x,S,X), 
\end{split}
\end{align}
where $\Upsilon'(.)$ is a real function converging to $0$ uniformly as $|h_\epsilon(e_n,\frac{\partial}{\partial s})| \to 0$ and $x\in A_k$. The last equality used that the mean curvature $H^\epsilon$ goes uniformly to $0$ by Lemma \ref{meancurv} and that $\|\nabla^\epsilon Z^\epsilon\|_{h_\epsilon}$ is bounded uniformly in $\epsilon>0$. For any $\hat{a}>0$, let  $\epsilon_k>0$ and $\hat{b}>0$ be small enough so that (\ref{consequence1}) is true with $\hat{r}=\hat{r}_k$ and $|\Upsilon'(b)|\leq \hat{a}$. The computation (\ref{long'}) and the fact that $Z^\epsilon=X^\epsilon_\parallel$ on $U_\epsilon\backslash\big(\gamma_\epsilon(\partial U_\epsilon \times [0,\delta_\epsilon])\big)$ imply that for $k$ large:
\begin{align*}
 \big|\int \divergence^{\epsilon_k} X^{\epsilon_k}_\parallel dV''_k \big|
& = \big| \int \divergence^{\epsilon_k}Z^{\epsilon_k} dV''_k\big|\\ 
& = \big|\int \divergence^{\epsilon_k}Z^{\epsilon_k} dV'_k\big|\\
& \leq \big|\int \chi_{\{|h_{\epsilon_k}(e_n,\frac{\partial}{\partial s})|\leq \hat{b}\}}\divergence^{\epsilon_k}_S Z^{\epsilon_k} dV'_k(x,S)  \big| + C\hat{a}\\
& \leq 2C\hat{a}.
\end{align*}
As previously, $C$ is a constant depending on $X$ but neither on ${\epsilon_k}$ nor on $\hat{a},\hat{b}$. The second equality above comes from the stationarity of $V_k$ for $h_{\epsilon_k}$. Thus combining with (\ref{limit1}) we obtain
$$\int \divergence^{0} X^{0}_\parallel dV''_\infty= \lim_{k\to \infty}\int \divergence^{\epsilon_k} X^{0}_\parallel dV''_k =  \lim_{k\to \infty}\int \divergence^{\epsilon_k} X^{\epsilon_k}_\parallel dV''_k =0$$
since $X^\epsilon_\parallel$ converges smoothly to $X^0_\parallel$. On the other hand taking $\epsilon\to 0$, (\ref{consequence1'})
and the minimality of $\partial U$ imply that 
$$\int \divergence^{0} X^{0}_\parallel dV'_\infty=0.$$
Finally we conclude that 
$$\int \divergence^{0} X^{0}_\parallel dV_\infty=\int \divergence^{0} X^{0}_\parallel dV'_\infty+\int \divergence^{0} X^{0}_\parallel dV''_\infty=0$$
and (\ref{separately2}) follows. Together with (\ref{separately1}), this proves (\ref{stat}), which ends the proof as explained in Step 2 and Step 3.

\end{proof}

\section{Proof of Theorem \ref{Yau}}

Let $(M,g)$ be a connected closed Riemannian manifold. In this section, all the minimal hypersurfaces considered are closed and smoothly embedded. We say that a two-sided minimal hypersurface is \textit{degenerate} if its Jacobi operator has a non-trivial kernel. If such a hypersurface is degenerate and stable, then the kernel of its Jacobi operator is spanned by a positive eigenfunction. Note that for a two-sided minimal hypersurface which is either unstable or non-degenerate stable, it is well-known that the hypersurface has a neighborhood foliated by closed leaves which, when not equal to the minimal hypersurface itself, have non-zero mean curvature vector. Actually a similar result is true for degenerate stable minimal hypersurfaces, as the following lemma shows.

\begin{lemme} \label{cmcfol}
Let $\Gamma$ be a two-sided degenerate stable minimal hypersurface in $(M,g)$ and $\nu$ a choice of unit normal vector on $\Gamma$. Then there exist a positive number $\delta_1$ and a smooth map $w:\Gamma \times (-\delta_1,\delta_1) \to \mathbb{R}$ with the following properties:
\begin{enumerate}
\item for each $x\in \Gamma$, we have $w(x,0)=0$ and $\phi_0:=\frac{\partial}{\partial t}w(x,t)|_{t=0} $ is a positive function in the kernel of the Jacobi operator of $\Gamma$,
\item for each $t\in(-\delta_1,\delta_1)$, we have $\int_\Gamma (w(.,t)-t\phi_0)\phi_0 =0 $,
\item for each $t\in(-\delta_1,\delta_1)$, the mean curvature of the hypersurface 
$$\{\exp(x,w(x,t) \nu(x)) ; x\in \Gamma\}$$
is either positive or negative or identically zero.
\end{enumerate}
\end{lemme}

\begin{proof}
The argument was already used in the proof of \cite[Proposition 5]{BrayBrendleNeves}. It is an application of the implicit function theorem. The mean curvature of $\{\exp(x,w(x,t) \nu(x)) ; x\in \Gamma\}$ is proportional to the positive eigenfunction $\phi_0$, which explains why either it has a sign or it is identically zero.
\end{proof}

If the minimal hypersurface $\Gamma$ is one-sided $(M,g)$, one can still apply the previous discussion (in particular Lemma \ref{cmcfol}) in a double-cover of $M$ where $\Gamma$ lifts to a two-sided hypersurface.

Let $S$ be a minimal hypersurface embedded in a compact $(n+1)$-dimensional compact manifold $\hat{M}$ (possibly with boundary). Let $\mu>0$, consider a neighborhood $\mathcal{N}$ of $S$ and a diffeomorphism 
$$F: \{x\in \hat{M}; d(x,S)\leq \mu\} \to \mathcal{N}$$
such that $F(x)=x$ for $x\in S$. We define the following:
\begin{itemize}
\item $S$ has a \textit{contracting neighborhood} if there are such $\mu$, $\mathcal{N}$ and $F$ such that for all $t\in (0,\mu]$, $F(\{x\in \hat{M}; d(x,S)=t\})$ has mean curvature vector pointing towards $S$,
\item$S$ has a \textit{expanding neighborhood} if there are such $\mu$, $\mathcal{N}$ and $F$ such that for all $t\in (0,\mu]$, $F(\{x\in \hat{M}; d(x,S)=t\})$ has mean curvature vector pointing away from $S$,
\item$S$ has a \textit{mixed neighborhood} if $S$ is two-sided, contained in the interior of $\hat{M}$ and there are such $\mu$, $\mathcal{N}$ and $F$ such that for all $t\in [-\mu,0)$ (resp. $t\in(0,\mu]$), $F(\{x\in \hat{M}; d_{signed}(x,S)=t\})$ has mean curvature vector pointing towards (resp. pointing away from) $S$. Here $d_{signed}$ is a choice of signed distance locally around $S$.
\end{itemize}
By the first variation formula, it is clear that if $S$ has a contracting (resp. expanding) neighborhood, the slices $F(\{x\in \hat{M}; d(x,S)=t\})$ have larger (resp. smaller) $n$-volume than $S$.

Given $(M,g)$ we will need to cut $M$ along some of its minimal hypersurfaces and get a compact manifold with boundary. Let $\Gamma$ be a connected minimal hypersurface in $M$. There are three possibilities: $\Gamma$ is either two-sided separating, or two-sided non-separating, or one-sided. In the first case, we cut $M$ along $\Gamma$, then choose any one of the two components $M\backslash \Gamma$: its metric completion is a compact manifold whose boundary has one connected component. In the second case, we cut $M$ along $\Gamma$ and obtain by completion a compact manifold whose boundary has two connected components. In the third case, we consider the metric completion of the complement of $\Gamma$ in $M$: it is a compact manifold whose boundary has one component. In all cases, let $M_1$ be the compact manifold we get. It is naturally endowed with a metric still denoted by $g$. We can repeat this procedure for any minimal hypersurface $\Gamma_1\subset (\interior{M_1},g)$, obtain $M_2$ and go on. Thus we construct a finite sequence $(M,g)$, $(M_1,g)$,... $(M_J,g)$ by successive cuts. Note that no two of them are isometric: each cut either adds new boundary components or strictly reduces the $(n+1)$-volume. We will say that a compact manifold $U$ \textit{is obtained by cutting $M$ along minimal hypersurfaces} if $U$ is the manifold obtained after a positive finite number of steps we just described.


\begin{lemme} \label{disjoint}
Let $(N,g')$ be a connected compact manifold, possibly with boundary. Suppose that the boundary, if non-empty, has a contracting neighborhood and that any minimal hypersurface in $N$ has a neighborhood which is either contracting or expanding or mixed. Then the following is true.
\begin{enumerate}
\item If the interior of $(N,g')$ contains two disjoint connected minimal hypersurfaces, then the interior of $N$ contains a minimal hypersurface with a contracting or mixed neighborhood.
\item If the interior of $(N,g')$ contains a minimal hypersurface with a contracting or mixed neighborhood, then one can cut $N$ along some minimal hypersurfaces and get a different manifold $(N',g')$ such that $\partial N'$ has a contracting neighborhood in $N'$.
\end{enumerate}
\end{lemme}
\begin{proof} 
First we prove $(1)$: let $\Gamma$, $\Gamma'$ be two disjoint minimal hypersurfaces. We can suppose that both $\Gamma$ and $\Gamma'$ have an expanding neighborhood. We cut $N$ along $\Gamma$, $\Gamma'$ and choose a component of $N \backslash (\Gamma \cup \Gamma')$ that has at least two different new boundary components coming from $\Gamma$, $\Gamma'$. We call this new manifold $N''$. Let $S_0$ be a component coming from $\Gamma$. We minimize its $n$-volume in its homological class inside of $N''$. By standard geometric measure theory and the maximum principle, we get a smooth minimal hypersurface and one of its component $S_1$ is two-sided and contained inside the interior of $N''$. The hypersurface $S_1$ was obtained by a minimization procedure so it is in particular $n$-volume minimizing under small smooth deformations. Hence it has a contracting neighborhood.

To prove (2), let $\Gamma$ be a connected minimal hypersurface with a contracting or mixed neighborhood, inside the interior of $N$. In the first case, we just cut along $\Gamma$ and obtain the desired $N'$. If $\Gamma$ has a mixed neighborhood, then by definition it is two-sided. When $\Gamma$ is separating $N$ into two components, we cut along $\Gamma$ and one of the components has a boundary with contracting neighborhood as desired. Suppose now that $\Gamma$ is not separating. Then once again we cut along $\Gamma$ and get a manifold $\hat{N}$ but following that, we minimize in its homological class the $n$-volume of one of the two new boundary components coming from $\Gamma$. Let us call that component $S_2 \subset \partial \hat{N}$, and assume it is the one with an expanding neighborhood. The smooth minimal hypersurface  resulting from minimizing the $n$-volume in the homological class of $S_2$ has $n$-volume strictly less than $S_2$.
Thus necessarily, that minimal hypersurface has a component $S_3$ inside the interior of $\hat{N}$. As in the previous paragraph, $S_3$ can be chosen two-sided and with a contracting neighborhood. We can cut $N$ along this hypersurface $S_3$, and throw away a connected component if necessary, in order to get the desired $N'$.

\end{proof}

Next, suppose by contradiction that $(M,g)$ only contains finitely many minimal hypersurfaces. Then by Lemma \ref{cmcfol} each one of them has either a contracting neighborhood or an expanding neighborhood or a mixed neighborhood. Let us cut $M$ along minimal hypersurfaces a certain number of times in such a way that we get a new manifold $U$ whose boundary, if not empty, has a contracting neighborhood. We choose $U$ so that the number of cuts is maximal: it is possible since the cuts are realized along minimal hypersurfaces, which are in finite number in $M$. By Lemma \ref{disjoint}, we know that all minimal hypersurfaces embedded in the interior of $U$ have an expanding neighborhood and any two such hypersurfaces intersect. By the work of Marques and Neves \cite{MaNeinfinity}, the manifold $(M,g)$ also necessarily contains at least two disjoint minimal hypersurfaces, otherwise it would contain infinitely many minimal hypersurfaces. Thus we deduce from Lemma \ref{disjoint} that $U$ is not equal to $M$ or in other words the boundary $\partial U $ is not empty.

The following lemma is true in general.          
\begin{lemme} \label{comparison}
Let $(\tilde{U},\tilde{g})$ be a connected compact manifold such that $\partial \tilde{U}$ is a minimal hypersurface with a contracting neighborhood. Let $\tilde{\Sigma}_1, \dots, \tilde{\Sigma}_q$ be the connected components of $\partial \tilde{U}$. Assume that every minimal hypersurface in the interior of $\tilde{U}$ has an expanding neighborhood and that any two of them intersect. Then for any minimal hypersurface $\Gamma$ in the interior of $\tilde{U}$:
\begin{enumerate}
\item if $\Gamma$ is two-sided,
$$\mathcal{H}^n(\Gamma)> \max\{\mathcal{H}^n(\tilde{\Sigma}_1),\dots,\mathcal{H}^n(\tilde{\Sigma}_q)\},$$
\item if $\Gamma$ is one-sided,
$$2\mathcal{H}^n(\Gamma)> \max\{\mathcal{H}^n(\tilde{\Sigma}_1),\dots,\mathcal{H}^n(\tilde{\Sigma}_q)\}.$$
\end{enumerate}
\end{lemme}
\begin{proof}
We consider the metric completion $\mathbf{C}$ of $\tilde{U}\backslash \Gamma$. It has at least one new boundary component $S$ coming from $\Gamma$. The hypersurface $\Gamma$ cannot be two-sided and non-separating in $\tilde{U}$. Otherwise we minimize the $n$-volume of $S$ in its homological class in $\mathbf{C}$. By the maximum principle, one component of the resulting minimal hypersurface is contained in the interior of $\mathbf{C}$: in particular it does not intersect $\Gamma$ in $\tilde{U}$, which contradicts our assumption. Suppose now that $\Gamma$ is separating in $\tilde{U}$. Then $\mathbf{C}$ has two connected components $\mathbf{C}_1$, $\mathbf{C}_2$. Each of them has a boundary component isometric to $\Gamma$. Suppose that $S$ (resp. $S'$) is such a component in $\partial \mathbf{C}_1$ (resp. $\partial \mathbf{C}_2$). We minimize the $n$-volume of $S$ in its homological class in $\mathbf{C}_1$. We get a resulting minimal hypersurface $\tilde{S}\subset\mathbf{C}_1$. Since $\tilde{S}$ does not touch $S$ by the maximum principle and since the interior of $\mathbf{C}_1$ cannot contain minimal hypersurfaces, $\tilde{S}$ is contained in $\partial \mathbf{C}_1\backslash S$. It shares the same homology class with $S$, so actually 
$$\tilde{S} =\partial \mathbf{C}_1\backslash S \text{ and so } \mathcal{H}^n(\Gamma) > \mathcal{H}^n(\partial \mathbf{C}_1\backslash S).$$
The same argument applied to $S'\subset  \mathbf{C}_2$ gives $$\mathcal{H}^n(\Gamma) > \mathcal{H}^n(\partial \mathbf{C}_2\backslash S).$$
On the other hand, $\big(\partial \mathbf{C}_1\backslash S\big) \cup\big(\partial \mathbf{C}_2\backslash S' \big) \text{ is isometric to } \partial \tilde{U}$ so we readily obtain that  
$$\mathcal{H}^n(\Gamma) >\max\{\mathcal{H}^n(\tilde{\Sigma}_1),\dots,\mathcal{H}^n(\tilde{\Sigma}_q)\}.$$
In the case where $\Gamma$ is one-sided in $\tilde{U}$, we argue similarly but this time with the unique new boundary component $S$ of the connected manifold $\mathbf{C}$, and $S$ is isometric to a double cover of $\Gamma$. This proves the lemma.
\end{proof}

Since we are supposing by contradiction that $M$ contains finitely many minimal hypersurfaces, let $\Gamma_1,\dots,\Gamma_k$ be the minimal hypersurfaces contained in the interior of $U$. 
Let $\Sigma_1,\dots,\Sigma_q$ be the connected components of $\partial U$, which is non-empty as previously explained. Suppose that 
$$\mathcal{H}^n(\Sigma_1)\geq \mathcal{H}^n(\Sigma_j) \text{ for all } j\in\{1,\dots,q\}.$$
Then by Lemma \ref{comparison} and by the construction of $U$, for all $i\in \{1,\dots,k\}$ we have: 
\begin{align}\label{plusgrand}
\begin{split}
\mathcal{H}^n(\Gamma_i)& >\mathcal{H}^n(\Sigma_1) \text{ if $\Gamma_i$ is two-sided},\\
2\mathcal{H}^n(\Gamma_i) & >\mathcal{H}^n(\Sigma_1) \text{ if $\Gamma_i$ is one-sided}.
\end{split}
\end{align}

Let $\mathcal{C}(U)$ as constructed in Subsection \ref{construction}. By Theorem \ref{minmax} and since any two of the $\Gamma_i$ intersect, the widths $\omega_p=\omega_p(\mathcal{C}(U))$ are realized as an integer multiple of the $n$-volume of one of the hypersurfaces $\Gamma_1,\dots,\Gamma_k$:
\begin{equation}\label{realize}
\forall p, \exists i_p\in\{1,\dots,k\}, \exists m_{p}\in\{1,2,3,\dots\}, \quad \omega_p = m_{p} \mathcal{H}^n(\Gamma_{i_p}),
\end{equation}
and $m_{p}$ is even if $\Gamma_{i_p}$ is one-sided.
We would like to derive a contradiction from the asymptotic behavior of the widths $\omega_p$ described in Theorem \ref{asymptotic} and inequality (\ref{plusgrand}). The next lemma is an elementary arithmetic result.
\begin{lemme} \label{arithmetic}
Let $\alpha_1,\dots,\alpha_R$ be a collection of real numbers strictly larger than $1$. Consider a sequence of increasing positive real numbers $\{u_p\}_{p\geq 1}$ such that for all $p\geq1$:
\begin{itemize}
\item $u_1\geq1$, $u_{p+1}\geq u_p+1$ and
\item $u_p\in \{m\alpha_r; m\in\{1,2,3,\dots\}, r\in\{1,\dots,R\}\}.$
\end{itemize}
Then there exists an $\bar{\epsilon}_0=\bar{\epsilon}_0(\alpha_1,\dots,\alpha_R)>0$ such that for $p$ large enough,
$$u_p>(1+\bar{\epsilon}_0)p.$$

\end{lemme}
\begin{proof}
We can assume that $(\alpha^{-1}_1,\dots,\alpha^{-1}_S)$ (where $S\leq R$) form a $\mathbb{Q}$-basis of the $\mathbb{Q}$-vector space generated by $\alpha^{-1}_1,\dots,\alpha^{-1}_R$. Suppose that  
\begin{equation}\label{linearcomb}
\forall r\in\{S+1,\dots,R\}, \exists c_{r,1},\dots,c_{r,S} \in \mathbb{Z}, \quad c_r \alpha^{-1}_r= \sum_{i=1}^S c_{r,i}\alpha^{-1}_i.
\end{equation}
Since $1,\frac{\alpha_1}{\alpha_2},\dots,\frac{\alpha_1}{\alpha_S}$ are independent real numbers over $\mathbb{Q}$, by Weyl's equidistribution theorem, the sequence 
$$\{(l\frac{\alpha_1}{\alpha_2},\dots,l\frac{\alpha_1}{\alpha_S}); l\in \mathbb{N}\}$$
is equidistributed modulo $\mathbb{Z}^{S-1}$. In particular, if ${\epsilon}\in(0,1)$, then the density of the set 
$$\mathcal{L}:=\{l\in \mathbb{N}; \forall i\in\{2,\dots,S\}, \exists q\in \mathbb{N}, |l\frac{\alpha_1}{\alpha_i}-q|< \frac{\epsilon}{2}\}$$
exists and is equal to $\mathbf{d}(\mathcal{L})=\epsilon^{S-1}$, where the density $\mathbf{d}(\mathcal{S})$ of a discrete subset $\mathcal{S}$ of $\mathbb{R}$ is defined (when it exists) by
$$\mathbf{d}(\mathcal{L}):=\lim_{x\to \infty}\frac{\sharp\{l\in\mathcal{L}; l\leq x\}}{x}.$$
Now note that by (\ref{linearcomb}), for all $l\in \mathcal{L}$ and for all $r\in \{S+1,\dots,R\}$, there is an integer $q(l,r)$ such that
$$|c_{r}l \frac{\alpha_1}{\alpha_{r}} -q(r,l)| < (|c_{r,2}|+\dots+|c_{r,S}|)\epsilon.$$
Thus, if we define $M:=\prod_{r=S+1}^{R} c_{r}$, there is a constant $C$ depending only on $\alpha_1,\dots,\alpha_R$ such that for all $l\in \mathcal{L}$ and $r\in\{2,\dots,R\}$, there is an integer $q'(l,r)$ satisfying
\begin{equation} \label{proche}
|Ml \alpha_1 - q'(l,r)\alpha_r| < C\epsilon.
\end{equation}
Define 
$$M\alpha_1\mathcal{L}:=\{Ml\alpha_1 ; l\in \mathcal{L}\}.$$ We can compute the density of this discrete subset of $\mathbb{R}$:
$$\mathbf{d}(M\alpha_1\mathcal{L}) = \frac{\epsilon^{S-1}}{M\alpha_1}.$$
Let us define 
$$\mathcal{A}:=\{m\alpha_r; m\in\{1,2,3,\dots\}, r\in\{1,\dots,R\}\}.$$ Note that by assumption $\{u_p\}_{p\geq 1} \subset \mathcal{A}$. Fix $\epsilon_1\in(0,1)$ such that $\min_{r\in\{1,\dots,R\}}\alpha_r > 1+3\epsilon_1.$ We can choose $\epsilon<\epsilon_1/C$ and then (\ref{proche}) gives for all $l'\in M\alpha_1\mathcal{L}$:
\begin{equation} \label{formul}
[l'+\epsilon_1, l'+2\epsilon_1+1] \cap \mathcal{A} =\varnothing.
\end{equation}
For a positive real number $P$, consider the integer 
$$\bar{n}(P):=\sharp\{l'\in M\alpha_1\mathcal{L} ; l'+2\epsilon_1 \leq P \}.$$ Note that 
\begin{equation}\label{alo}
\lim_{P\to \infty}\frac{\bar{n}(P)}{P} = \mathbf{d}(M\alpha_1\mathcal{L}).
\end{equation}
From (\ref{formul}), (\ref{alo}) and the fact that the gap between any two consecutive $u_p$, $u_{p+1}$ is at least $1$, we deduce that for all $P\in \mathbb{R}$ large enough: 
\begin{align*}
 \sharp\{ p ; u_p \leq P \}  & \leq {P - \bar{n}(P)\epsilon_1}\\
 & \leq P(1-\mathbf{d}(M\alpha_1\mathcal{L})\frac{\epsilon_1}2).
 \end{align*}
 Consequently for $p\in \mathbb{N}$ large enough, $$u_{p} \geq (p-1) \frac{1}{1-\mathbf{d}(M\alpha_1\mathcal{L})\frac{\epsilon_1}2}$$
so the lemma is proved by taking for instance
$$\bar{\epsilon}_0:=\mathbf{d}(M\alpha_1\mathcal{L})\frac{\epsilon_1}{4}>0.$$

\end{proof}

Of course by rescaling, we can assume that $\mathcal{H}^n(\Sigma_1)=1$. Let us apply the previous lemma to the volume spectrum $\{\omega_p\}_{p\geq1}$ (see (\ref{realize}) and Theorem \ref{asymptotic} ({1})), and the $n$-volumes (resp. twice the $n$-volumes) of the two-sided (resp. one-sided) hypersurfaces $\Gamma_i$ (see (\ref{plusgrand})). We obtain for a certain $\bar{\epsilon}_0>0$ and all $p$ large enough:
$$\omega_p>(1+\bar{\epsilon}_0)p.$$
But Theorem \ref{asymptotic} (2) implies that $ \lim_{p\to \infty} p^{-1}\omega_p=1$ and in particular for $p$ large enough,
$$\omega_p<(1+\bar{\epsilon}_0)p.$$
These two inequalities give the desired contradiction. 

\begin{remarque} \label{remarque}
The arguments in this section actually yield a more precise version of Theorem \ref{Yau}: any compact manifold $M^{n+1}$ ($2\leq n \leq 6$) whose boundary has a contracting neighborhood in $M$ contains infinitely many closed embedded minimal hypersurfaces.
\end{remarque}

\section*{Appendix}

We show the monotonicity formula, Lemma \ref{monotonicity}, by slightly editing the standard proof in the Euclidean case which can be found in \cite[Chapter 4 Section 17, Chapter 8 Section 40]{Simon} for instance.

\begin{proof}[Proof of Lemma \ref{monotonicity}]
We fix $\xi \in B_1$ and define $g_{\xi}$ to be the flat metric on $B_3$ obtained by the constant $2$-tensor $\tilde{g}(\xi)$ on $B_3$. Note that while $g_{\xi}$ is flat, it may be different from $g_{eucl}$. However since by hypothesis $\|\tilde{g}-g_{eucl}\|_{C^1(\Omega)}\leq \eta$, we do have 
\begin{equation} \label{compar}
(1-\eta) g_{eucl} \leq g_{\xi}\leq (1+\eta)g_{eucl}.
\end{equation}
Let $r:B_3\to \mathbb{R}$ be the distance function to $\xi$ for the metric $g_\xi$, i.e. $r(x):=\|x-\xi\|_{g_\xi}$. Let $\tilde{\nabla}$ be the Levi-Civita connection of $\tilde{g}$, let $\tilde{\divergence}_S$ be the divergence along an $n$-plane $S$ computed with $\tilde{\nabla}$ and let $\tilde{\nabla}_S $ be the $\tilde{g}$-orthogonal projection of $\tilde{\nabla}$ on $S$. Let $h:B_3\to [0,1]$ be a smooth nonnegative function. Let $X$ be the compactly supported smooth vector field defined by
$$X_x=\gamma(r) h \frac{1}{2}\tilde{\nabla} r^2,$$
where $\gamma:\mathbb{R}\to [0,1]$ is a smooth function with 
$$\gamma'(t)\leq 0 \text{ }\forall t, \gamma(t)\equiv1 \text{ for } t\leq \rho, \gamma(t)\equiv0 \text{ for } t\geq 1.$$
From the fact that $V$ is $\tilde{g}$-stationary, we know that
\begin{equation}\label{lop}
\int \tilde{\divergence}_S X dV(x,S)=0.
\end{equation}
Let $(x,S)$ be an $n$-plane at a point $x\neq \xi$ and let $e_1,\dots,e_n$ be an orthonormal basis of $S$ for $\tilde{g}$. We compute
\begin{align}\label{lop1}
\begin{split}
\tilde{\divergence}_SX = & \sum_{i=1}^n \tilde{g}(\tilde{\nabla}_{e_i}X,e_i) \\
= & r\gamma'(r)h \|\tilde{\nabla}_S r\|^2_{\tilde{g}}  + r \gamma(r) \tilde{g}(\tilde{\nabla} r, \tilde{\nabla}_S h) +\gamma(r) h \sum_{i=1}^n \tilde{g}(\tilde{\nabla}_{e_i} (r\tilde{\nabla}r), e_i) \\
\geq & r\gamma'(r) h \|\tilde{\nabla}_S r\|^2_{\tilde{g}}  + r \gamma(r) \tilde{g}(\tilde{\nabla} r, \tilde{\nabla}_S h) + n\gamma(r) h -Cr\gamma(r) h\\
\geq & r\gamma'(r) h \|\tilde{\nabla}_S r\|^2_{\tilde{g}}  + r \gamma(r) \tilde{g}(\tilde{\nabla} r, \tilde{\nabla}_S h) + n\gamma(r) h \|\tilde{\nabla} r\|^2_{\tilde{g}} -2Cr\gamma(r) h\|\tilde{\nabla} r\|^2_{\tilde{g}}
\end{split}
\end{align}
Here $C$ is a positive constant depending only on $\eta$ and going to $0$ as $\eta\to 0$. The last inequality is obtained by a Taylor expansion at $\xi$ for instance ($\tilde{g} = g_\xi$ at $\xi$).
Now let us consider $\epsilon\in(0,1)$, and a smooth function $\varphi:\mathbb{R}\to [0,1] $ such that $\varphi(t)=1$ for $t\leq 1$, $\varphi(t)=0$ for $t\geq 1+\epsilon$ and $\varphi'(t)\leq 0$ for all $t$. Then we take $\gamma(r)=\varphi(r/\rho)$, assuming $(1+\epsilon)\rho<2$. Since
$$r\gamma'(r) = -\rho \frac{\partial}{\partial \rho}[\varphi(r/\rho)],$$ 
formula (\ref{lop}) and inequality (\ref{lop1}) yield
\begin{align*}
& n\int  \gamma h  \|\tilde{\nabla} r\|^2_{\tilde{g}} dV - \rho \frac{\partial}{\partial \rho} \int \varphi(r/\rho) h \|\tilde{\nabla}_S r\|^2_{\tilde{g}} dV(x,S) \\
& \leq - \int r \gamma(r)  \tilde{g}(\tilde{\nabla} r, \tilde{\nabla}_S h) dV(x,S) + 3C\rho \int \varphi(r/\rho) h\|\tilde{\nabla} r\|^2_{\tilde{g}}  dV.
\end{align*}
We define 
$$\tilde{I}(\rho):=\int   \varphi(r/\rho) h \|\tilde{\nabla} r\|^2_{\tilde{g}} dV.$$
Denoting by $\tilde{\nabla}^{\perp} r$ the difference $\tilde{\nabla}r-\tilde{\nabla}_S r$, we have the identity $\|\tilde{\nabla} r\|^2_{\tilde{g}}=\|\tilde{\nabla}_S r\|^2_{\tilde{g}}+\|\tilde{\nabla}^{\perp} r\|^2_{\tilde{g}}$. By the previous inequality:
\begin{align*}
n\tilde{I}(\rho) - \rho\frac{\partial}{\partial\rho}\tilde{I}(\rho) \leq & -\rho \frac{\partial}{\partial \rho} \int \varphi(r/\rho) h \|\tilde{\nabla}^{\perp} r\|^2_{\tilde{g}}dV(x,S) \\
 &- \int r \gamma(r)  \tilde{g}(\tilde{\nabla} r, \tilde{\nabla}_S h) dV(x,S) +3C\rho\tilde{I}(\rho)
\end{align*}
which means, after rearrangements, 
\begin{align*}
&\frac{\partial}{\partial \rho} \Big[\exp(3C\rho) \rho^{-n} \tilde{I}(\rho)\Big]\geq \\ & 
\exp(3C\rho)\rho^{-(n+1)}\Big[\rho \frac{\partial}{\partial \rho} \int \varphi(r/\rho) h \|\tilde{\nabla}^{\perp} r\|^2_{\tilde{g}}dV(x,S)+\int r\gamma(r)  \tilde{g}(\tilde{\nabla} r, \tilde{\nabla}_S h) dV(x,S)\Big].
\end{align*}
Denote by $B'(\xi,r)$ the ball of radius $r$ for $g_\xi$ centered at $\xi$. Integrating the last inequality from $\sigma$ to $\rho$ and letting $\epsilon\to0$,  we get
\begin{align*}
&\exp(3C\rho) \rho^{-n} \int_{B'(\xi,\rho)} h\|\tilde{\nabla} r\|^2_{\tilde{g}} d\|V\|  - \exp(3C\sigma) \sigma^{-n} \int_{B'(\xi,\sigma)} h\|\tilde{\nabla} r\|^2_{\tilde{g}} d\|V\|  \\ &+  \exp(3C\rho) \int_\sigma^\rho \tau^{-n}\int_{B'(\xi,\tau)\times \mathbf{G}(n+1,n)}  \|\tilde{\nabla}r\|_{\tilde{g}} \|\tilde{\nabla}_Sh\|_{\tilde{g}} \quad dV(x,S) d\tau \\
&\geq   \exp(3C\rho) \int_{\big(B'(\xi,\rho)\backslash B'(\xi,\sigma)\big)\times \mathbf{G}(n+1,n)}  h \frac{\|\tilde{\nabla}^{\perp} r\|^2_{\tilde{g}}}{r^n}dV(x,S),
\end{align*}
which is almost what we want. To end the proof, observe that by (\ref{compar}) there is a constant $c'$ closer and closer to $0$ as $\eta\to 0$ such that
$$B(\xi,\sigma) \subset B'(\xi,(1+c')\sigma) \subset B'(\xi,(1+c')\rho)\subset B(\xi,(1+c')^2\rho),$$
$$1\leq (1+c')\|\tilde{\nabla} r\|^2_{\tilde{g}} \leq (1+c')^2 \text{ and } \|\tilde{\nabla}_Sh\|_{\tilde{g}} \leq (1+c')|{\nabla}_Sh|,$$
where $|{\nabla}_Sh|$ is the norm of the gradient of $h$ along $S$ computed with $g_{eucl}$ and the previous inequality (which is valid for all $0<\sigma\leq \rho<1$) implies
\begin{align*}
&\exp(3C\sigma) \sigma^{-n} \int_{B(\xi,\sigma)} h d\|V\| \\
&\leq  (1+c')\exp(3C\sigma) \sigma^{-n} \int_{B'(\xi,(1+c')\sigma)} h \|\tilde{\nabla} r\|^2_{\tilde{g}} d\|V\|\\
&\leq   (1+c')\exp(3C(1+c')\rho) \Big[\rho^{-n} \int_{B'(\xi,(1+c')\rho)} h \|\tilde{\nabla} r\|^2_{\tilde{g}} d\|V\| \\
&\quad + \int_{(1+c')\sigma}^{(1+c')\rho} \tau^{-n} \int_{B'(\xi,\tau)\times \mathbf{G}(n+1,n)} \|\tilde{\nabla}r\|_{\tilde{g}} \|\tilde{\nabla}_Sh\|_{\tilde{g}} \quad dV(x,S) d\tau\\
& \leq (1+c')\exp(3C(1+c')\rho) \Big[\rho^{-n} (1+c')\int_{B(\xi,(1+c')^2\rho)} h d\|V\| \\
& \quad + (1+c')^{n+1}\int_{(1+c')^2\sigma}^{(1+c')^2\rho} \bar{\tau}^{-n} \int_{B(\xi,\bar{\tau})\times \mathbf{G}(n+1,n)} |{\nabla}_Sh| \quad dV(x,S)d\bar{\tau}\Big] \\
& \leq (1+c')^{n+2}\exp(3C(1+c')\rho) \Big[\rho^{-n} \int_{B(\xi,(1+c')^2\rho)} h d\|V\|\\
& \quad + \int_{\sigma}^{(1+c')^2\rho} \bar{\tau}^{-n} \int_{B(\xi,\bar{\tau})\times \mathbf{G}(n+1,n)} |{\nabla}_Sh| \quad dV(x,S)d\bar{\tau}\Big].
\end{align*}
So the lemma is proved by taking $\frak{c}=3C$ and $\frak{a} = (1+c')^{n+2}-1$.

\end{proof}

\bibliographystyle{plain}
\bibliography{biblio19_08_30bis}

\end{document}